\pdfoutput=1
\documentclass{article}

\usepackage{xparse}
\usepackage{etoolbox}
\usepackage{kvoptions}

\usepackage{fix-cm}
\usepackage[T1]{fontenc}
\usepackage[utf8]{inputenc}
\usepackage[full]{textcomp}
\usepackage[l2tabu, orthodox]{nag}

\usepackage{lmodern} 
\usepackage[bitstream-charter,cal=cmcal]{mathdesign}

\usepackage{microtype}
\usepackage{csquotes}

\usepackage{graphicx}
\usepackage{tikz}
\usetikzlibrary{matrix, arrows, decorations.markings, decorations.pathmorphing, calc}

\tikzset{->-/.style={decoration={
  markings,
  mark=at position 0.5 with {\arrow{stealth}}},postaction={decorate}}}
\tikzset{->>-/.style={decoration={
  markings,
  mark=at position 0.5 with {\arrow{>>}}},postaction={decorate}}}
\tikzset{snake it/.style={decorate, decoration=snake}}

\usepackage{tikz-cd}

\usepackage{caption}
\usepackage{subcaption}
\usepackage{multirow}
\usepackage{booktabs}
\usepackage{tabularx}
\usepackage{tabulary}
\usepackage{longtable}
\usepackage{pdflscape}
\usepackage{rotating}

\newcolumntype{R}{>{$}r<{$}}
\newcolumntype{C}{>{$}c<{$}}
\newcolumntype{L}{>{$}l<{$}}

\usepackage{dsfont}
\usepackage{xfrac}
\usepackage{mathtools}
\usepackage{mleftright}
\usepackage{manfnt}
\usepackage{accents}

\usepackage{emptypage}
\usepackage{comment}

\usepackage[inline]{enumitem}
\usepackage[obeyFinal,colorinlistoftodos=true,textsize=footnotesize]{todonotes}

\usepackage[hyphens]{url}

\usepackage{hyperref}
\usepackage{geometry} 
\usepackage{amsthm}
\usepackage{thmtools,thm-restate}
\usepackage[capitalize]{cleveref}

\definecolor{dark-red}{rgb}{0.4,0.15,0.15}
\definecolor{dark-blue}{rgb}{0.15,0.15,0.4}
\definecolor{medium-blue}{rgb}{0,0,0.5}
\hypersetup{
    colorlinks, linkcolor={dark-red},
    citecolor={dark-blue}, urlcolor={medium-blue}
}

\LetLtxMacro{\amsmathdots}{\dots}
\usepackage{ellipsis}
\AtBeginDocument{\LetLtxMacro{\dots}{\amsmathdots}}

\let\ker\relax
\DeclareMathOperator{\ker}{Ker}

\DeclareMathOperator{\Aut}{Aut}

\DeclareMathOperator{\Hom}{Hom}

\DeclareMathOperator{\Sym}{Sym}

\DeclareMathOperator{\PGL}{PGL}

\newcommand*{\HS}{\text{HS}}

\newcommand*{\dpunct}[1]{\,\text{#1}}

\newcommand{\from}{\vcentcolon}

\newcommand{\dsum}{\oplus}
\newcommand{\dSum}{\bigoplus}
\newcommand{\tensor}{\otimes}

\NewDocumentCommand\xDeclarePairedDelimiter{mmm}
 {%
  \NewDocumentCommand#1{som}{%
   \IfNoValueTF{##2}
    {\IfBooleanTF{##1}{#2##3#3}{\mleft#2##3\mright#3}}
    {\mathopen{##2#2}##3\mathclose{##2#3}}%
  }%
 }

\xDeclarePairedDelimiter{\abs}{\lvert}{\rvert}
\xDeclarePairedDelimiter{\norm}{\lVert}{\rVert}
\xDeclarePairedDelimiter{\floor}{\lfloor}{\rfloor}
\xDeclarePairedDelimiter{\ceil}{\lceil}{\rceil}
\xDeclarePairedDelimiter{\gen}{\langle}{\rangle}
\xDeclarePairedDelimiter{\pseries}{\llbracket}{\rrbracket}
\xDeclarePairedDelimiter{\oneto}{[}{]}
\xDeclarePairedDelimiter{\parenth}{(}{)}

\NewDocumentCommand{\set}{somm}{%
   \IfNoValueTF{#2}
    {\IfBooleanTF{#1}{\{#3 \mid #4\}}{\mleft\{ #3 \mathrel{}\middle\vert\mathrel{} #4 \mright\}}}
    {\mathopen{#2\{}#3 \mathrel{}#2\vert\mathrel{} #4\mathclose{#2\}}}%
  }
\NewDocumentCommand{\present}{somm}{%
   \IfNoValueTF{#2}
    {\IfBooleanTF{#1}{\langle#3 \mid #4\rangle}{\mleft\langle#3 \mathrel{}\middle\vert\mathrel{} #4 \mright\rangle}}
    {\mathopen{#2\langle}#3 \mathrel{}#2\vert\mathrel{} #4\mathclose{#2\rangle}}%
  }
\NewDocumentCommand{\inner}{somm}{%
   \IfNoValueTF{#2}
    {\IfBooleanTF{#1}{\langle#3 , #4\rangle}{\mleft\langle#3 , #4 \mright\rangle}}
    {\mathopen{#2\langle}#3 , #4\mathclose{#2\rangle}}%
  }

\let\epsilon\varepsilon

\newcommand{\CC}{\mathbb{C}}

\newcommand{\LL}{\mathbb{L}}

\newcommand{\PP}{\mathbb{P}}
\newcommand{\QQ}{\mathbb{Q}}

\newcommand{\ZZ}{\mathbb{Z}}

\newcommand{\cA}{\mathcal{A}}
\newcommand{\cC}{\mathcal{C}}

\newcommand{\cH}{\mathcal{H}}

\newcommand{\cL}{\mathcal{L}}

\newcommand{\cO}{\mathcal{O}}

\newcommand{\cV}{\mathcal{V}}

\newcommand{\fm}{\mathfrak{m}}

\newcommand{\fS}{\mathfrak{S}}
\newcommand{\fX}{\mathfrak{X}}

\newcommand{\hol}{\mathrm{hol}}

\newcommand{\Symgr}{\Sym_{\mathrm{gr}}}
\newcommand{\id}{\mathrm{id}}

\newcommand*{\widebar}[1]{\mkern 1.5mu\overline{\mkern-1.5mu#1\mkern-1.5mu}\mkern 1.5mu}

\newcommand*{\cl}[1]{
\begingroup
    \setbox\z@=\hbox{\ensuremath{#1}}%
    \ifdimgreater{\wd\z@}{4em}{\mleft(#1\mright)^{-}}{\widebar{#1}}
\endgroup
}

\newcommand*{\interior}[1]{
\begingroup
    \setbox\z@=\hbox{\ensuremath{#1}}%
    \ifdimgreater{\wd\z@}{1.5em}{\mleft(#1\mright)^{\circ}}{\accentset{\circ}{#1}}
\endgroup
}

\newcommand*{\tolabel}[1]{\xrightarrow{\;#1\;}}
\newcommand{\isomto}{\xrightarrow{\,\;\smash{\raisebox{-0.4ex}{\ensuremath{\scriptstyle\cong}}}\;\,}}

\newcommand{\subto}{\hookrightarrow}
\newcommand{\isom}{\cong}
\newcommand{\equivto}{\xrightarrow{\,\;\smash{\raisebox{-0.6ex}{\ensuremath{\simeq}}}\;\,}}
\newcommand{\lequivto}{\xleftarrow{\,\;\smash{\raisebox{-0.6ex}{\ensuremath{\simeq}}}\;\,}}
\newcommand{\defeq}{\mathrel{\vcentcolon=}}

\numberwithin{equation}{section}
\declaretheorem[title=Theorem, refname={Theorem,Theorems}, Refname={Theorem,Theorems}]{maintheorem}

\declaretheorem[sibling=equation]{theorem}
\declaretheorem[sibling=theorem]{lemma}

\declaretheorem[sibling=theorem]{corollary}
\declaretheorem[sibling=theorem]{proposition}

\declaretheorem[sibling=theorem,refname={Conjecture,Conjectures}, Refname={Conjecture,Conjectures}]{conjecture}

\declaretheorem[numbered=no, title=Theorem]{theorem*}
\declaretheorem[numbered=no, title=Corollary]{corollary*}
\declaretheorem[numbered=no, title=Lemma]{lemma*}
\declaretheorem[numbered=no, title=Proposition]{proposition*}
\declaretheorem[numbered=no, title=Conjecture]{conjecture*}

\declaretheorem[sibling=theorem, style=definition, qed=$\lozenge$]{definition}
\declaretheorem[numbered=no, style=definition, title=Definition, qed=$\lozenge$]{definition*}

\declaretheorem[sibling=theorem, style=remark, qed=$\lozenge$]{remark}
\declaretheorem[numbered=no, style=remark, title=Remark, qed=$\lozenge$]{remark*}
\declaretheorem[sibling=theorem, style=remark, qed=$\lozenge$]{example}
\declaretheorem[numbered=no, style=remark, title=Example, qed=$\lozenge$]{example*}

\usepackage{sseq}

\usepackage[style=ieee-alphabetic, url=false, backref=true]{biblatex}

\renewbibmacro*{doi+eprint+url}{%
\iffieldundef{journaltitle}{
  \iftoggle{bbx:eprint}
    {\usebibmacro{eprint}}
    {}%
  \newunit\newblock
  \iftoggle{bbx:url}
    {\usebibmacro{url+urldate}}
    {}}{}}

\newcommand*{\Gr}{\mathrm{Gr}}

\DeclareMathOperator{\F}{F}
\DeclareMathOperator{\Z}{Z}

\DeclareMathOperator{\ev}{ev}

\newcommand{\cdga}{\textsc{cdga}}

\crefname{equation}{}{}
\Crefname{equation}{Equation}{Equations}

\newcommand{\MHSgr}{\text{MHS}^{\text{gr}}}

\newcommand{\bC}{\mathbb C}

\newcommand{\bL}{\mathbb L}

\newcommand{\bP}{\mathbb P}
\newcommand{\bQ}{\mathbb Q}
\newcommand{\bR}{\mathbb R}

\newcommand{\bZ}{\mathbb Z}

\newcommand{\lra}{\longrightarrow}

\newcommand{\catTop}{\mathsf{Top}}
\newcommand{\catF}{\mathsf{F}}

\newcommand{\CCH}{\check{H}_{c}}

\newcommand{\op}{\text{op}}
\newcommand{\Sing}{\mathrm{Sing}}
\newcommand{\Gammahol}{\Gamma_{\hol}}
\newcommand{\Gammacon}{\Gamma_{\cC^0}}
\DeclareMathOperator*{\colim}{colim}

\DeclareMathOperator{\map}{map}

\newcommand\longmapsfrom{\mathrel{\reflectbox{\ensuremath{\longmapsto}}}}

\newtheorem{construction}[theorem]{Construction}
\newtheorem{convention}[theorem]{Convention}

\DeclareMathOperator{\APL}{A_{PL}}

\newcommand{\diff}{\mathrm{d}}

\title{Homological stability for the space of hypersurfaces with marked points}

\author{Alexis Aumonier \and Ronno Das}

\begin{document}
\maketitle

\begin{abstract}
We study the space of smooth marked hypersurfaces in a given linear system.
Specifically, we prove a homology h-principle to compare it with a space of sections of an appropriate jet bundle.
Using rational models, we compute its rational cohomology in a range of degrees, and deduce a homological stability result for hypersurfaces of increasing degree. 
We also describe the Hodge weights on the stable cohomology, and thereby connect our topological result to motivic results of Howe.
\end{abstract}

\tableofcontents

\section{Introduction}

Consider a connected smooth complex projective variety $X$ with a line bundle $\cL$.
Let $U(\cL) \subset \Gammahol(\cL)$ be the open subset of \emph{non-singular} holomorphic global sections of $\cL$ and consider the incidence variety:
\[
    Z(\cL) = \set{(f, x) \in U(\cL) \times X}{f(x) = 0} \subset \Gammahol(\cL) \times X \dpunct.
\]
If we take the sections modulo scalar multiples, ie quotient by the action of $\CC^\times$, we get the space of \emph{$\cL$-hypersurfaces} given by $U(\cL)/\CC^\times$ and the \emph{universal non-singular $\cL$-hypersurface} given by $Z(\cL)/\CC^\times$.
In another language, these spaces are respectively the open locus of smooth divisors inside the complete linear system $|\cL|$ and the restriction of the universal flat family above it.
In rational homology the $\CC^\times$ often behaves like a direct factor, see \cref{scalar-multiple}, and we will not worry about this quotient in the rest of the paper.

More generally, let $\Z^r(\cL)$ be the bundle over $U(\cL)$ whose fiber over $f$ is $\F^r(\cV(f))$, the configuration space of $r$ distinct points on the vanishing locus $\cV(f)$ of $f$, topologized as a subspace of $U \times X^r$.
Explicitly,
\[
\Z^r(\cL) = \set[\big]{(f, x_1, \dotsc, x_r) \in \Gammahol(\cL) \times X^r}{f \in U(\cL),\ f(x_i) = 0,\ x_i \ne x_j \text{ for } i \ne j}
\]
is the \emph{space of $\cL$-hypersurface sections with $r$ marked points}.

\begin{maintheorem}\label{maintheorem:stability}
Suppose $\cL$ is ample and fix $r \ge 1$.
Then for each $i$ and for $d$ sufficiently large\footnote{It suffices to take $d > \max(\abs*{\chi(X)}, k(2i + 2r + 3))$, with $k$ such that $\cL^{\tensor k}$ is very ample, by combining the bounds in \cref{maintheorem:comparison,euler-class-bound}.} we have an isomorphism
\[H^i(\Z^r(\cL^{\tensor d}); \QQ) \isom H^i(A_r(X, c_1(\cL))) \dpunct,\]
where $A_r(X, c_1(\cL))$ is a commutative differential graded algebra that is \emph{independent of $d$} and explicitly described in \cref{stable-cdga}.
This isomorphism is $\fS_r$ equivariant and preserves cup products (when $d$ is appropriately large).
In particular, $H^i(\Z^r(\cL^{\tensor d}); \QQ)$ stabilizes with $d$, as an $\fS_r$-representation.
\end{maintheorem}

The above theorem follows a series of results on geometric and arithmetic stability of the family $\Z^r(\cL)$ as $\cL$ is made increasingly ample, notably the main theorems in \cite{poonen_bertini_2004}, \cite{vakil_discriminants_2015} and \cite{howe_motivic_2019}.
Nearer to our cohomological results is the computation of the rational cohomology of $U(\cL)$ in a range of stability, which can and should be thought of as the case $r = 0$, for $X = \PP^n$ by Tommasi \cite{tommasi_stable_2014}, for $X$ a curve by O.~Banerjee \cite[Section~4]{banerjee_filtration_2019}, and for general $X$ in \cite{aumonier_h-principle_2022} and \cite{das_cohomological_2022}.

The classes in $H^i(\Z^r(\cL^{\tensor d}); \bQ)$ for different $d$ that are identified by these isomorphisms also have the same Hodge weights (see \cref{weights-on-stable-cohomology}).
The consequent stabilization of Hodge--Euler characteristics (\cref{hodge-euler-characteristic}) was observed by Howe \cite{howe_motivic_2019}.

\begin{remark}
In fact, the \cdga{} $A_r(X, c_1(\cL))$ only depends on the integer $r$, the rational cohomology ring of $X$, its Poincaré pairing and the projectivized Chern class $[c_1(\cL)] \in \bP H^2(X)$. 
\end{remark}

\begin{remark}
Under the assumptions of \cref{maintheorem:stability} we do not have an example where the cohomology ring $H^*(A_r(X, c))$ actually depends on the class $c$.
However the construction of $A_r(X, c)$ does not require $c$ to be represented by an element in the ample cone.
Without this restriction $H^*(A_r(X, c))$ can in fact depend on $c$, see \cref{non-ample-c1-dependence}.
\Cref{one-point-stability} also shows that the cohomology ring $H^*(A_1(X, c))$ \emph{does not} depend on $c$ when $r = 1$.

In general, the dependence of $A_r(X, c)$ on $c$ is only in terms of its differential, ie as an algebra $A_r(X, c)$ does not depend on $c$.
This is reflected in the observation that the analogs of \cref{maintheorem:stability} in terms of point counts and classes in the Grothendieck ring (see \cite[Remark~1.15(i)]{vakil_discriminants_2015} and \cite{howe_motivic_2019}) have no such dependence, since these invariants are akin to the Euler characteristic.
\end{remark}

\begin{remark}
The description of $A_r(X, c_1(\cL))$ is amenable to explicit computations; see \cref{sec:computations}.
\end{remark}

There are no obvious maps between $\Z^r(\cL^{\tensor d})$ for varying $d$ and the stability in \cref{maintheorem:stability} is (to our knowledge) not induced by such maps of spaces.
Nevertheless, the isomorphism in \cref{maintheorem:stability} is not just an abstract isomorphism once we ascribe more geometric meaning to $A_r(X, c)$. In doing so, we also manage to uncover the Hodge weights in the cohomology of $\Z^r(\cL^{\tensor d})$, as we shall explain below in \cref{hodge-euler-characteristic}. 

Our proof of \cref{maintheorem:stability} involves identifying $A_r(X, c)$ as a rational model for a continuous analog of $\Z^r(\cL^{\otimes d})$.
Explicitly, let $J^1(\cL) \cong \cL \oplus \Omega^1(\cL)$ be the first-order jet bundle of $\cL$ and define
\[
    \Z^r_{\cC^0}(\cL) \defeq \set{(s,x_1,\dotsc,x_r) \in \Gamma_{\cC^0}(J^1(\cL)) \times X^r}{s(x) \neq 0 \ \forall x ,\ s(x_i) \in \Omega^1(\cL),\ x_i \ne x_j \text{ for } i \ne j}
\]
by imitating the construction of $\Z^r(\cL)$ above while replacing $U(\cL)$ by the space $\Gamma_{\cC^0}(J^1(\cL)-0)$ of non-vanishing continuous sections; see \cref{sec:jet-bundles,sec:section-spaces} for precise definitions.
There is a \emph{jet-expansion map}
\[
    j^1 \from \Gammahol(\cL) \to \Gammacon(J^1(\cL))
\]
which induces a map $\Z^r(\cL) \to \Z^r_{\cC^0}(\cL)$ by the restriction of $(j^1 \times \id) \from \Gammahol(\cL) \times X^r \to \Gammacon(J^1(\cL)) \times X^r$, which we will also call a jet-expansion map and continue to denote by $j^1$.

Recall that if $\cL$ is very ample then $\cL^{\tensor d}$ is $d$-jet ample (see \cref{defn:jet-ample} of jet ampleness).
Then the following theorem, combined with the computation of $H^*(\Z^r_{\cC^0}(\cL); \QQ)$ from \cref{serre-ss-degeneration} and \cref{rational-model-stabilizes}, refines \cref{maintheorem:stability}:
\begin{maintheorem}\label{maintheorem:comparison}
Let $r \ge 1$ and $\cL$ be $d$-jet ample for some $d > 2i + 2r + 3$. Then the jet-expansion map $j^1$ induces an isomorphism
\[
    H_i(\Z^r(\cL); \ZZ) \isomto H_i(\Z^r_{\cC^0}(\cL); \ZZ)
\]
in integral homology.
\end{maintheorem}

Although \cref{maintheorem:comparison} is an integral statement, we focus on the more manageable rational computations in this paper. Thus for the rest of the introduction we will use cohomology with rational coefficients but will suppress it from the notation for the sake of brevity and readability.

Theorem~1.1 in \cite{aumonier_h-principle_2022} should be thought of as the $r = 0$ case of \cref{maintheorem:comparison}, comparing just the space $U(\cL)$ of non-singular sections, without any markings, with the analogous continuous section space $\Gamma_{\cC^0}(J^1(\cL)-0)$.
Given this result and the definition of $\Z^r(\cL)$, it seems reasonable to use the Serre spectral sequence of the fiber bundle $\Z^r(\cL) \to U(\cL)$, which for $r = 1$ in fact degenerates on the $E_2$ page (by Deligne's degeneration theorem, see \cref{deligne-degeneration} for an application).
The terms of this spectral sequence are given by the cohomology of $U(\cL)$ with certain \emph{local coefficients}, with stalk $H^*(\cV(f))$ at $f \in U(\cL)$.
However, the usual strategy for tackling the cohomology of $U(\cL)$ (in \cite{aumonier_h-principle_2022} and elsewhere) is to pass to the complement in $\Gammahol(\cL)$ using Alexander duality and any non-trivially twisted coefficients on $U(\cL)$ cannot possibly extend to this contractible space.

In contrast, what actually lets us make headway is the other projection $\Z^r(\cL) \to \F^r(X)$ to the configuration space of $X$.
While we do get a Serre spectral sequence for the fibration $\Z^r_{\cC^0}(\cL) \to \F^r(X)$, on the algebraic side the map $\Z^r(\cL) \to \F^r(X)$ may not be\footnote{At least not in general, excepting when $\Aut(X)$ acts transitively on $\F^r(X)$.} a fibration.
It is however a \emph{microfibration} (see \cref{defn:microfibration}).
In \cref{sec:microfibration} we use this fact to reduce the proof of \cref{maintheorem:comparison} to a comparison of the fibers of $\Z^r(\cL)$ and $\Z^r_{\cC^0}(\cL)$ above $\F^r(X)$. 
To be more precise, we establish a fiberwise homology isomorphism (in a range of degrees) in \cref{h-principle-prescribed-derivatives}, after replacing the map $\Z^r(\cL) \to \F^r(X)$ by the map $\Z^r(\cL) \to (\Omega^1_X \tensor \cL)^r|_{\F^r(X)}$,
\[(f, x_1, \dots, x_r) \mapsto (\mathrm{d}f(x_1), \dotsc, \mathrm{d}f(x_r)) \dpunct,\]
that records not only the marked points but also the (necessarily non-zero) derivatives of the section at each of these points.

\subsection{Marking a single point}
\label{sec:one-point}

For $r = 1$, the space $\Z^1(\cL) = Z(\cL)$ is the ``universal $\cL$-hypersurface''.
This case has a mildly different flavor than $r > 1$, partly because the fiber $\cV(f)$ is projective.
In this case the computation of $H^*(A_1(X, c))$ is relatively simple and we get the following explicit description, whose proof we will defer to \cref{one-point-stability-proof}.

\begin{corollary}\label{one-point-stability}
Suppose $\cL$ is $d$-jet ample. Writing $\Omega^1(\cL) - 0$ for the complement of the zero section in $\Omega^1(\cL) = \Omega^1_X \tensor \cL$, we have the following isomorphism in the range $* < \frac{d - 5}{2}$:
\begin{equation*}
H^*(Z(\cL); \bQ) \isom H^*\big(\Omega^1(\cL) - 0; \bQ\big) \tensor \Symgr\big(H^{*<2n}(X)[-1]\big) \dpunct.
\end{equation*}
\end{corollary}
Here, and in what will follow, $n$ is the complex dimension of $X$ and $H^{*<2n}(X)[-1]$ is the graded vector space $\dSum_{i = 0}^{2n - 1} H^i(X)[-1]$, ie classes in $H^i(X)$ for $i < 2n$ are placed in degree $i + 1$.

\begin{remark}
Since $Z(\cL) \to U(\cL)$ is a smooth projective bundle, Deligne's degeneration theorem (see \cite{deligne_theoreme_1968}) implies that $H^*(U(\cL)) \to H^*(Z(\cL))$ must be injective.
We can verify explicitly that each algebra generator of $\Symgr(H^*(X)[-1])$ maps to a non-zero class under the identifications in \cref{one-point-stability}, ie in the range where both \cref{maintheorem:stability,one-point-stability} apply. 
For a generator $sb \in H^{*<2n}(X)[-1]$ corresponding to a shift of a $b \in H^{*<2n}(X)$, its image is also a free generator, so is non-zero by definition.
On the other hand, the shifted top class $s[X] \mapsto \tau - \epsilon$, where $\tau \in H^{2n + 1}(\Omega^1(\cL) - 0)$ is the class whose image under fiber integration (or the Gysin homomorphism) is $c_1(\cL) \in H^2(X)$ and $\epsilon$ is as in the proof of \cref{one-point-stability}. 
In particular, each term is non-zero and they lie in distinct direct summands.
\end{remark}

In the range $* < 2n - 2$ (and for $d$ sufficiently large) \cref{one-point-stability} follows from Nori's connectivity theorem \cite{nori_algebraic_1993}.
Taking $X = \PP^n$, we recover the stabilization (and stable cohomology) obtained by I.~Banerjee in \cite{banerjee_stable_2023}.

\begin{corollary}[cf {\cite[Theorem~1.1]{banerjee_stable_2023}}]
If $X = \PP^n$, $r = 1$ and $\cL = \cO(d)$ then 
\[H^*(Z(\cL); \QQ) \isom H^*(\PP^{n-1} \times U(\cL); \QQ)\]
in the range of degrees $* < \frac{d-5}{2}$.
\end{corollary}

\begin{proof}
Using that $H^*(U(\cL)) \isom \Symgr(H^*(X)[-1])$ under these assumptions, it suffices to identify 
\[H^*(\Omega^1(\cL) - 0) \isom H^*(\PP^{n-1}) \tensor \Symgr(\QQ\gen*{t}) \dpunct,\]
where $|t| = 2n + 1$, for instance by directly computing the cohomology of the \cdga{} 
\[\QQ[x_2, \alpha_{2n-1}]/x^{n+1}\dpunct, \quad dx = 0\dpunct,\ d\alpha = x^n \dpunct,\]
where the subscripts denote the degrees of the respective generators.
\end{proof}

\begin{remark}
Our range is slightly worse that the one given in \cite[Theorem~1.1]{banerjee_stable_2023}, namely $* < \frac{d-1}{2}$, though it could potentially be improved by running our proof of \cref{maintheorem:stability} in the special case $X = \bP^n$. 
\end{remark}

\begin{remark}
\Cref{one-point-stability} gives us an exact criterion for when the stable cohomology of $Z(\cL)$ is finite dimensional: $H^i(Z(\cL))$ vanishes for $i$ large and $\cL$ sufficiently jet ample (depending on $i$) if and only if $H^*(X)$ is concentrated in even degrees.
For instance, this holds if $X$ is $\PP^n$ or a Grassmannian but fails if $X$ is a curve of positive genus.
In contrast, for $r > 1$ the stable cohomology is necessarily non-zero in infinitely many degrees for any $X$. This is already visible in the weightwise Euler characteristic described in \cref{hodge-euler-characteristic} and was noted in \cite{howe_motivic_2019} as a required feature.
\end{remark}

The fiber bundle $\pi \colon Z(\cL) \to U(\cL)$ produces a family $\cH^q \defeq R^q\pi_*\bQ$ of locally constant sheaves over $U(\cL)$, one for each $q \ge 0$ and with stalk $H^q(\cV(f))$ over $f \in U(\cL)$.
By the Lefschetz hyperplane theorem, $\cH^q$ is constant for $q \ne n - 1$, and for $q < n - 1$ is isomorphic to the constant sheaf $H^i(X)$ via the map induced by the inclusion $Z(\cL) \subset U(\cL) \times X$.
On the other hand, $\cH^{n - 1}$ is usually not constant.
The Serre spectral sequence of $Z(\cL) \to U(\cL)$ has signature
\begin{equation} \label{deligne-degeneration-formula}
E_2^{p, q} = H^p(U(\cL); \cH^q) \implies H^*(Z(\cL)) \end{equation}
and degenerates immediately by Deligne's degeneration theorem.
Define the stable (twisted) Betti numbers 
\[
    \rho_p \defeq \lim_{d \to \infty} \dim_\QQ H^p(U(\cL^{\tensor d}); \cH^{n - 1}) \dpunct.
\]
where $\cL$ is any ample line bundle.
Equating the Poincaré series of the two sides of \cref{deligne-degeneration-formula} and plugging in our computation of $H^*(Z(\cL^{\tensor d}))$, we get:
\begin{corollary}\label{deligne-degeneration}\pushQED{\qed}
Let $P(t)$ be the Poincaré series of $\Symgr(H^{*<2n}(X)[-1])$ and let $\beta_i = \dim H^i(X)$ be the Betti numbers of $X$.
Then
\[\sum_{p \ge 0} \rho_p t^p = P(t) \left[\sum_{i = 0}^{n} \beta_{i + n - 1} t^i - \sum_{i = 1}^{n - 1} \beta_{i + n + 1}t^i + \sum_{i = n+1}^{2n+1} \beta_{i - n} t^i - \sum_{i = n + 2}^{2n} \beta_{i - n - 2} t^i\right] \dpunct. \qedhere\]
\popQED
\end{corollary}
Note that 
\[P(t) = \prod_{i = 0}^{2n - 1} (1 - (-t)^{i + 1})^{(-1)^i \beta_i}\]
gives an explicit description of the right hand side as a rational function in $t$ with coefficient determined by $\beta_i$.
Further, since $\beta_i \le \beta_{i + 2}$ for $i \le n - 1$ and $\beta_i \ge \beta_{i + 2}$ for $i \ge n - 1$ (by the hard Lefschetz theorem), the second factor has non-negative coefficients.

\begin{remark}
When $X = \bP^n$, the sum on the right simplifies to either $0$ if $n$ is even, or $1 + t^{2n+1}$ if $n$ is odd. In either case, the stable cohomology of $U(\cL)$ with coefficients in the vanishing cohomology $\cH^{n-1}_\mathrm{van}$ vanishes. This recovers \cite[Corollary~1.3]{banerjee_stable_2023}.
\end{remark}

\subsection{Unordered marked points}
While the configuration spaces in the rest of the paper are ordered, let us deal with the unordered case here.
The symmetric group $\fS_r$ acts on $\Z^r(\cL) \subset \Gammahol(\cL) \times \F^r(X)$ by permuting the coordinates of $\F^r(X)$.
The map $\Z^r(\cL) \to U(\cL)$ descends to the quotient: the map $\Z^r(\cL)/\fS_r \to U(\cL)$ is also a fiber bundle, now with fiber $\F^r(\cV(f))/\fS_r$, the \emph{unordered} configuration space of $\cV(f)$, over $f \in U(\cL)$.
The analog of \cref{maintheorem:comparison} holds with the same proof, or with rational coefficients we can just use the transfer isomorphisms
\[H^*(\Z^r(\cL)/\fS_r; \QQ) \isom H^*(\Z^r(\cL); \QQ)^{\fS_r} \dpunct.\]
In particular we have the following analog of \cref{maintheorem:stability}:
\begin{corollary}
If $\cL$ is ample and $d$ is sufficiently large then
\[H^i(\Z^r(\cL^{\tensor d})/\fS^r; \QQ) \isom H^i(A_r(X, c)^{\fS_r}) \dpunct.\]
In particular, $\dim H^i(\Z^r(\cL^{\tensor d})/\fS_r; \QQ)$ stabilizes for large $d$. \qed
\end{corollary}
Similarly, taking the quotient by other subgroups of $\fS_r$ produces analogous results for various ``colored'' configuration spaces, ie bundles over $U(\cL)$ with fiber intermediate covers of $\F^r(\cV(f))/\fS_r$.

\subsection{Weightwise Euler characteristic}
\label{hodge-euler-characteristic}

Since $\Z^r(\cL)$ and $U(\cL)$ are quasiprojective varieties, their rational cohomology rings are canonically equipped with mixed Hodge structures.
The mixed Hodge structure on the stable cohomology of $H^*(U(\cL); \QQ)$ is compatible with its description as $\Symgr(H^*(X)[-1](-1))$, ie in the stable range the vector space isomorphism 
\[\Symgr(H^*(X)[-1](-1)) \lra H^*(U(\cL); \QQ)\]
is an isomorphism of mixed Hodge structures, see \cite[Proposition~8.6]{aumonier_h-principle_2022}.
Here the $(-1)$ denotes a Tate twist, in particular if $b_j \in H^*(X)$ has weight $w$ then $sb_j$ in the notation of \cref{sec:recollections-sectionspace} has weight $w + 2$. 
Similarly, the Kriz--Totaro \cdga{} $C_r(X)$ for $\F^r(X)$ (see \cref{construction:totaro-model}), due to its identification with the Leray spectral sequence of the algebraic inclusion $\F^r(X) \subto X^r$, has a mixed Hodge structure which is compatible with that on $H^*(\F^r(X))$, \emph{after passing to the associated graded for the weight filtration}\footnote{In general only after passing to the associated graded, see \cite[Section~3]{looijenga_torelli_2020}.}.

In \cref{rational-model-stabilizes,stable-cdga} we explain that the \cdga{} $A_r(X, c)$ in \cref{maintheorem:stability} has the form
\begin{equation}\label{stable-cdga-tensor-factorization}
A_r(X, c) \defeq C_r(X) \tensor \Symgr(H^*(X)[-1](-1)) \tensor \Symgr(\QQ \gen*{\alpha_1, \dots, \alpha_r, \eta_1, \dots, \eta_r}) \dpunct,
\end{equation}
with $\alpha_i$ of degree $2n - 1$ corresponding to the fundamental class of the punctured cotangent space $T^*_{x_i}X - 0$ at the marked points $x_i \in X$, and $\eta_i$ of degree $2n$ corresponding to the cohomology generator of the loop space $\Omega(T^*_{x_i}X \dsum \CC - 0)$.
Let $K_0(\text{MHS})$ be the Grothendieck ring of mixed Hodge structures, write $K_0(\MHSgr)$ for its (cohomology) graded version, and denote by $\LL = [H^2(\PP^2)] \in K_0(\MHSgr)$, which is of pure weight $2$ and lies in (cohomological) degree $2$.
Based on their geometric meaning, the natural Hodge structure to ascribe to the generators $\alpha_i$ and $\eta_i$ are: $\QQ\gen*{\alpha_i} \isom \LL^n[1]$ and $\QQ\gen*{\eta_i} \isom \LL^{n+1}[2]$.
With this choice the differentials in \cref{stable-cdga} preserve mixed Hodge structures so we can construct $A_r(X, c)$ in the category of mixed Hodge structures.

\begin{conjecture} \label{mhs-conjecture}
With the mixed Hodge structure defined above, the isomorphisms in \cref{maintheorem:stability} are isomorphisms of mixed Hodge structure after passing to the associated graded for the weight filtration.
\end{conjecture}

Ostensibly, a proof of \cref{mhs-conjecture} should involve carrying out the computation in \cref{rational-model-proof} in an appropriate category of \cdga{}s equipped with mixed Hodge structures.
However, such a lift of the computation is beyond the scope of this project.
Instead, we show that if we assign weights to the generators of $A_r$ as predicted by \cref{mhs-conjecture} then we obtain the correct weights on $H^*(\Z^r(\cL))$.

\begin{theorem} \label{weights-on-stable-cohomology}
Consider the \cdga{} $A_r(X, c)$ from \cref{stable-cdga} as a quotient of
\[H^*(X)^{\tensor r}\tensor \Symgr(H^*(X)[-1] \dsum \QQ\gen*{G_{ab}, \alpha_i, \eta_i}) \dpunct,\]
where $1 \le i, a, b \le r$ and $a \ne b$.
In addition to the grading by degrees $|G_{ab}| = |\alpha_i| = 2n - 1$ and $|\eta_i| = 2n$, put an additional  ``weight grading'' $w$ on the generators so that 
\[w(G_{ab}) = w(\alpha_i) = 2n \dpunct, \qquad w(\eta_i) = 2n + 2\]
and for $b \in H^i(X)$, 
\[w(b) = i \dpunct, \qquad w(sb) = i + 2 \dpunct,\]
where $sb \in H^*(X)[-1]$ is the corresponding generator. 
Extend multiplicatively to make $A_r(X, c)$ a bigraded algebra.
Then its differential preserves the additional grading (ie $w(da) = w(a)$ for ``weight-homogeneous'' $a$), and the induced grading on $H^*(A_r(X, c))$ agrees with the canonical weight grading of $H^*(Z^r(\cL))$ under the isomorphisms in \cref{maintheorem:stability}.
\end{theorem}

We defer the proof of this theorem to \cref{cdga-weights-proof}. 
But using it we can compute the weightwise Euler characteristic of $A_r(X, c)$.
Since classes of a given weight can only appear in finitely many degrees, by inspection of the weights ascribed to the generators above in this special case or by Deligne's proof of the Weil conjectures in general, this also computes the weightwise Euler characteristic of $\Z^r(\cL)$ for $\cL$ sufficiently ample.

To be precise, for $w \ge 0$ let
\[\chi(r, k) := \chi \left[\Gr^W_k H^*(A_r (X, c))\right] = \sum_{i \ge 0} (-1)^i \dim \Gr^W_k H^i (A_r(X, c)) \dpunct,\]
and note that this sum is finite, ie $\Gr^W_k H^i$ is non-zero for only finitely many $i$ given $k$. 
As is usual for Euler characteristic, it also does not depend on differentials, and since $A_r(X, c)$ is degreewise finite dimensional we can compute it as 
\[\chi(k) = \sum_{i \ge 0} (-1)^i \dim \Gr^W_k (A_r(X, c))^{\deg = i}\]
It will be convenient to directly compute the generating function $P_r(w) = \sum_k \chi(r, k) w^k$.
Set $P_{\F^r}$ and $P_U$ to be the similarly defined generating functions for the weightwise Euler characteristics of $\F^r(X)$ and $\Symgr H^*(X)[-1](-1)$ respectively.

Then from \cref{stable-cdga-tensor-factorization} we get
\[P_r(w) = P_{\F^r}(w) P_U(w) (1 - w^n)^r (1 + w^{2n + 2} + w^{2(2n + 2)} + \dotsb)^r = P_{\F^r}(w) P_U(w) \left(\frac{1 - w^{2n}}{1 - w^{2n + 2}}\right)^r\]
since each $\alpha_i$ is in odd degree and of weight $2n$ and each $\eta_i^j$ is in even degree and of weight $j(2n + 2)$.
Since $X$ is projective, $H^i(X)$ lies purely in weight $i$, so we can also write down a formula for $P_U$ in terms of the Betti numbers $\beta_i$ of $X$:
\[P_U(w) = \prod_i (1 - w^{i + 1})^{(-1)^i \beta_i} \dpunct.\]

If \cref{mhs-conjecture} holds then performing the same computation in $K_0(\MHSgr)\llbracket w \rrbracket$ we get
\[P_r(w) = P_{\F^r}(w) P_U(w) \left(\frac{1 - (\LL w^2)^n}{1 - (\LL w^2)^{n + 1}}\right)^r \dpunct.\]
Now the variable $w$ is in a sense redundant, since the coefficient of $w^k$ is exactly the weight $k$ part of the coefficients, except to allow ``formal'' infinite sums. 
In our case we can omit $w$ if we instead take the completion with respect to $\LL$ and this produces exactly the formula from \cite[Example~1.3.5]{howe_motivic_2019}.
In this sense \cref{mhs-conjecture} would be a cohomological lift of those results of \cite{howe_motivic_2019} taking values in $K_0(\text{MHS})$ and related rings.
To be precise, assuming \cref{mhs-conjecture} and some mild application of the representation theory of $\fS_d$, we can derive the following theorem, where $\chi_\HS$ is the above lift of $\chi$ to $K_0(\text{MHS})$.
\begin{theorem}[{\cite[Theorem~A]{howe_motivic_2019}}] \label{howe-stabilization}
If $\cL$ is very ample and $V$ is a finite-dimensional $\fS_r$ representation, then $\chi_\HS(F_r(\cL^{\tensor d}); V)$ stabilizes weightwise to $\chi_\HS(A_r(X, c) \tensor_{\fS_r} V)$ as $d \to \infty$.
\end{theorem}

However any of the results in this paper, or their tentative analogs in étale cohomology, say nothing about the results of \cite{howe_motivic_2019} involving point counts over finite fields: although the Grothendieck--Lefschetz trace formula does say that such point counts share many properties with Euler characteristic, it a priori involves cohomology of all degrees simultaneously so the cohomological input is never contained in our range of stability.

\subsection{Outline}
In \cref{sec:setup} we define our main objects of interest and recall some elementary facts about sections of line bundles and their jets. 
\Cref{sec:comparison,sec:h-principle} are devoted to the proof of \cref{maintheorem:comparison}, with the ``high-level'' construction involving microfibrations in \cref{sec:comparison} and the necessary fiberwise homology isomorphism established in \cref{sec:h-principle}.
In \cref{sec:rational-model} we compute the rational cohomology of $Z^r(\cL)$ via a \cdga{} model, thereby deducing \cref{maintheorem:stability} and the remaining corollaries.
\Cref{sec:computations} in particular lists the computation of $H^*(A_r(X, c_1(\cL)))$ for some specific examples, such as $X = \PP^1, \PP^2, \PP^1 \times \PP^1$.

\subsection*{Acknowledgements}

We are grateful to Oscar Randal-Williams for pointing out the relevance of microfibrations, which are a crucial ingredient in our arguments.
We thank Søren Galatius, Florian Naef and Dan Petersen for helpful conversations throughout various stages of the project.
Thanks to Nils Prigge for pointing us towards and explaining the Grivel--Halperin--Thomas construction of the Serre spectral sequence.
AA would also like to thank Nathalie Wahl, Michael Weiss, and Kai Cieliebak, comprising his PhD thesis committee, for spotting a mistake in an earlier version.

AA was supported by the Danish National Research Foundation through the Copenhagen Centre for Geometry and Topology (DNRF151) as well as the European Research Council (ERC) under the European Union’s Horizon 2020 research and innovation programme (grant agreement No.~682922).
RD was supported by the European Research Council (ERC) under the European Union's Horizon 2020 research and innovation programme (grant agreement No.~772960), by the Danish National Research Foundation through the Copenhagen Centre for Geometry and Topology (DNRF151), as well as by Dan Petersen's Wallenberg Academy fellowship during various parts of the project.

\section{Configuration spaces and jet bundles}
\label{sec:setup}

This section is meant to be a self-contained recollection of relevant facts about configuration spaces and jet bundles followed by their specialization to our context. 
As in the introduction, fix $X$ to be a connected smooth complex projective variety of complex dimension~$n$. 
We write $\Gammahol(-)$ for the space of holomorphic global sections of a bundle on $X$.

\subsection{Configuration spaces}
For any topological space $S$, let $\F^r(S)$ be the configuration space of $r$ ordered points in $S$
\[
    \F^r(S) \defeq \set{(x_1, \dots, x_r) \in S^r}{x_i \ne x_j \text{ for } i \ne j} \dpunct.
\]
The symmetric group $\fS_r$ acts on $\F^r(S)$ by permuting coordinates. More generally, for $Z \to S$, define the fiberwise configuration space
\[
    \F^r_S(Z) \defeq \F^r(Z) \cap Z^{\times_S r},
\]
where $Z^{\times_S r}$ denotes the $r$-fold fiber product $Z \times_S Z \times_S \dotsm \times_S Z$.
Denoting the fiber of $Z$ over $s \in S$ by $Z_s$, we have a natural identification
\[\F^r_S(Z) = \set{(s, z_1, \dots, z_r)}{(z_1, \dots, z_r) \in \F^r(Z_s)} \subset S \times \F^r(Z) \subset S \times Z^r \dpunct.\]
In particular, the projection map $\F^r_S(Z) \to S$ has fiber $\F^r(Z_s)$ over $s \in S$.
For convenience, we identify $\F^1_S(Z) = Z$ and $\F^0_S(Z) = S$.

\subsection{Jet bundles}
\label{sec:jet-bundles}

In this work, we use jet bundles to suitably talk about derivatives of sections of vector bundles. For the unfamiliar reader, we offer a minimal overview of the general theory developed in \cite{grothendieck_elements_1967}.

\bigskip

Let $\cL$ be a holomorphic line bundle on $X$. Its bundle of first order jets, defined in \cite[Section~16.7]{grothendieck_elements_1967}, will be denoted by $J^1(\cL)$. It is a holomorphic vector bundle on $X$ of complex rank $\dim_\bC X + 1$ which fits in an exact sequence of holomorphic vector bundles
\begin{equation}\label{eqn:SES-jet-bundle}
    0 \lra \Omega^1(\cL) \lra J^1(\cL) \lra \cL \lra 0
\end{equation}
where $\Omega^1(\cL) = \Omega^1_X \otimes \cL$ is the cotangent bundle of $X$ tensored with $\cL$. Although this short exact sequence does not split in general (in the category of holomorphic vector bundles), it informally indicates that the jet bundle records the value (in $\cL$) and the first derivative (in $\Omega^1(\cL)$) of sections of $\cL$. More precisely:
\begin{definition}
The short exact sequence above splits after taking holomorphic global sections. Writing $\mathrm{d}f$ for the derivative of a global section $f$ of $\cL$, the morphism
\begin{align*}
	j^1 \colon \Gammahol(\cL) &\lra \Gammahol(J^1(\cL)) = \Gammahol(\cL) \oplus \Gammahol(\Omega^1(\cL)) \\
	f &\longmapsto (f, \diff f)
\end{align*}
is called the \emph{(first order) jet expansion}.
\end{definition}

\begin{example}\label{example:jet-projective-space}
Let us consider the case where $X = \bP^n$ and $\cL = \cO(k)$ with $k \geq 1$. The space of global sections $\Gammahol(\cO(k))$ can be identified with the complex vector space of homogeneous polynomials of degree $k$ in $n+1$ variables. Writing $z_0, \ldots, z_n$ for the variables, Euler's identity
\[
    \sum_i z_i \frac{\partial f}{\partial z_i} = k \cdot f
\]
shows that knowing the $n+1$ partial derivatives of a section $f$ in the homogeneous coordinates amounts to knowing the section. This fact can be leveraged to an isomorphism $J^1(\cO(k)) \isom \cO(k-1)^{\oplus n+1}$ of holomorphic vector bundles\footnote{For a proof, see \cite{di_rocco_line_1998}. Note however that such a decomposition is very peculiar to projective spaces.}. With these identifications, the jet expansion is given by
\[
    \Gammahol(\cO(k)) \ni \quad f \longmapsto \left(\frac{\partial f}{\partial z_0}, \ldots, \frac{\partial f}{\partial z_n} \right) \quad \in \Gammahol(\cO(k-1))^{\oplus n+1} \dpunct. \qedhere
\]
\end{example}

\begin{definition}\label{singular-section}
A global section $f \in \Gammahol(\cL)$ is said to be \emph{singular} at $x \in X$ if the first jet $j^1(s)(x) = 0$. It is called \emph{non-singular} if it is not singular at any $x \in X$.
\end{definition}

\begin{remark}
The vanishing locus of any non-zero global section $f \in \Gammahol(\cL)$ is the subvariety given by $\cV(f) \defeq \set{x \in X}{f(x) = 0} \subset X$. When $f$ is non-singular, $\cV(f)$ is a smooth subvariety by the implicit function theorem.
\end{remark}

\begin{example}
In the situation of \cref{example:jet-projective-space}, a global section $f$ is singular at a point $x \in \bP^n$ precisely when $\frac{\partial f}{\partial z_i}(x) = 0$ for each $i = 0, \ldots, n$. This is indeed the more classical Jacobian criterion.
\end{example}

The following property, jet ampleness, is the technical key to many arguments in this work:
\begin{definition}[Compare \cite{beltrametti_generation_1999}]\label{defn:jet-ample}
Let $\cL$ be a holomorphic line bundle on $X$. Let $k \geq 0$ be an integer. Let $x_1, \dotsc, x_t$ be $t$ distinct points in $X$ and $(k_1,\dotsc,k_t)$ be a $t$ tuple of positive integers such that $\sum_i k_i = k+1$. Write $\fm_i$ for the maximal ideal sheaf corresponding to $x_i$, and $\cL_{x_i}$ for the stalk of the sheaf $\cL$ at $x_i$. We say that $\cL$ is \emph{$k$-jet ample} if the evaluation map
\[
    \Gammahol\left(\cL\right) \lra \bigoplus_{i=1}^t \ \cL_{x_i} / \fm_i^{k_i} \cL_{x_i}
\]
is surjective for any $x_1,\dotsc,x_t$ and $k_1,\dotsc,k_t$ as above.
\end{definition}

Note that the image of $f \in \Gamma_\hol(L)$ in $\cL_{x_i}/\fm_i^{k_i}$ can be identified with $f(x_i)$ if $k_i = 1$ and with $j^1(f)(x_i) = (f(x_i), \diff f(x_i))$ if $k_i = 2$.
More generally and less formally, the definition insists that there is a section of $\cL$ with arbitrary prescribed Taylor approximations of total order $\le k + 1$ at any $t$ points of $X$. Although we are only concerned with first order approximations in this article and we could replace the condition with a weaker version where each $k_i \le 2$, we have chosen to state our results using the existing notion of jet ampleness for the sake of familiarity.

\subsection{Marked hypersurfaces and section spaces}
\label{sec:section-spaces}

We continue with the notation from the previous section. For a holomorphic bundle $E \to X$, recall that we denote the space of its holomorphic global sections by $\Gamma_\hol(E)$. We will write $\Gammacon(E)$ for the larger space of continuous global sections. Both section spaces are topologized as subspaces of the continuous mapping space $\map(X,E)$, the latter being endowed with the compact-open topology.

\bigskip

Let $U(\cL) \subset \Gammahol(\cL)$ be the subspace of non-singular sections. The incidence variety
\[
    Z(\cL) \defeq \set{(f, x)}{f(x) = 0} \subset U(\cL) \times X
\]
is equipped with projections to $U(\cL)$ and $X$; the fiber over $f \in U(\cL)$ is the smooth ``$\cL$-hypersurface'' $\cV(f) \subset X$.
\begin{definition}
For $r \geq 0$, the space of \emph{$\cL$-hypersurfaces with $r$ (ordered) marked points} is
\[
\Z^r(\cL) \defeq \F^r_{U(\cL)}(Z(\cL)) = \set{(f, x_1, \dotsc, x_r)}{f(x_i) = 0,\ x_i \ne x_j \text{ for } i \ne j} \subset U(\cL) \times X^r \dpunct. \qedhere
\]
\end{definition}

As before, we will identify $\Z^1(\cL) = Z(\cL)$ and $\Z^0(\cL) = U(\cL)$. 
The space $\Z^r(\cL)$ comes with two projection maps: to $U(\cL)$ and to $\F^r(X)$. The fiber of $\Z^r(\cL) \to U(\cL)$ over $f \in U(\cL)$ is the configuration space $\F^r(\cV(f))$. 

\begin{remark}\label{scalar-multiple}
As $\cV(f) = \cV(\lambda f)$ for any $\lambda \in \bC^\times$, it is perhaps more geometrically meaningful to consider quotients of $\Z^r(\cL)$ (including $U(\cL) = \Z^0(\cL)$) by $\bC^\times$. 
However, our constructions are easier without taking this quotient and in general the analogs of our results may not hold for these quotients.
However in the ``typical case'', eg when $\cL$ is sufficiently ample (\cite[Theorem~3.4]{gelfand_discriminants_1994}, see also \cite[Corollary~1.2]{gelfand_discriminants_1994} and its preceding discussion), the discriminant is given by a hypersurface and in that case the following lemma shows that this distinction is inessential when considering rational cohomology.
In particular the analog of \cref{maintheorem:stability} holds for $\Z^r(\cL)/\CC^\times$, but possibly with a worse stable range.
\end{remark}

\begin{lemma}\label{scalar-multiple-leray-hirsch}
For $r \ge 0$, let $\CC^\times$ act on $\Z^r(\cL) \subset U(\cL) \times \F^r(X)$ by scalar multiplication on $U(\cL)$ and trivially on $\F^r(X)$.
Then this action is free and when $U(\cL) \subset \Gammahol(\cL)$ is the complement of a hypersurface we have an isomorphism:
\[H^*(\Z^r(\cL); \QQ) \isom H^*(\Z^r(\cL)/\CC^\times \times \CC^\times; \QQ) \dpunct.\]
\end{lemma}
\begin{proof}
Suppose $U(\cL) \subset \Gammahol(\cL)$ is the complement of the hypersurface defined by a (necessarily homogeneous) polynomial $\Delta$.
Then for any $(f, \vec x) \in \Z^r(\cL)$, the composite
\[
    \CC^\times \tolabel{\lambda \mapsto \lambda \cdot (f, \vec x)} \Z^r(\cL) \to U(\cL) \tolabel{\Delta} \CC^\times
\]
is the map $\lambda \mapsto \Delta(\lambda f) = \lambda^{\deg \Delta} \Delta(f)$ of degree $\deg \Delta \neq 0$ and hence induces an isomorphism on rational cohomology.
The claim follows by applying the Leray--Hirsch theorem.
\end{proof}

The constructions above can be adapted and repeated in the setting of continuous sections. More precisely, we write
\[
    U_{\cC^0}(\cL) \defeq \Gamma_{\cC^0}(J^1(\cL) - 0) = \set{s \in \Gamma_{\cC^0}(J^1(\cL))}{s(x) \neq 0, \ \forall x \in X}.
\]
To define the analogue of the incidence variety $Z(\cL)$, we have to find a corresponding notion of vanishing at a point for sections of the jet bundle. Recall that the evaluation of a section $f \in \Gammacon(J^1(\cL))$ at a point $x \in X$ is an element
\[
f(x) = (f_1(x), f_2(x)) \in J^1(\cL)_x = \cL_x \oplus \Omega^1(\cL)_x
\]
where $(-)_x$ denotes the fiber at $x$. If $f = j^1(g)$, then $f_1(x) = g(x)$ is simply the value of $g$ at $x$. We take a cue from this situation and define
\[
    Z_{\cC^0}(\cL) \defeq \set{(s, x) \in U_{\cC^0}(\cL) \times X}{s(x) \in 0 \oplus \Omega^1(\cL)_x \subset \cL_x \oplus \Omega^1(\cL)_x } \dpunct.
\]
The following pullback square of topological spaces is then a direct consequence of the definitions:
\[\begin{tikzcd}
Z(\cL) \rar["j^1 \times \id"] \dar & Z_{\cC^0}(\cL) \dar\\
U(\cL) \rar["j^1"] & U_{\cC^0}(\cL)
\end{tikzcd}\]
More generally, denoting $\F^r_{U_{\cC^0}(\cL)}(Z_{\cC^0}(\cL))$ by $\Z^r_{\cC^0}(\cL)$, we get a pullback square:
\[\begin{tikzcd}
\Z^r(\cL) \rar["j^1 \times \id"] \dar & \Z^r_{\cC^0}(\cL) \dar\\
U(\cL) \rar["j^1"] & U_{\cC^0}(\cL)
\end{tikzcd}\]

\section{Comparison isomorphism}
\label{sec:comparison}

Having defined the spaces $\Z^r(\cL)$ and $\Z^r_{\cC^0}(\cL)$ above, we are now ready to state our main theorem about integral cohomology. Its proof will occupy the rest of this section.
\begin{theorem}\label{theorem:main-theorem}
Let $X$ be a connected smooth complex projective variety, $r \ge 1$ an integer, and $\cL$ a $d$-jet ample line bundle on $X$. Then the map 
\[
    j^1 \times \id \colon \Z^r(\cL) \lra \Z^r_{\cC^0}(\cL)
\]
induces an isomorphism in integral homology in the range of degrees $* < \frac{d-3}{2}-r$.
\end{theorem}

We now proceed towards the proof of the theorem, beginning with some notation. Given a fiber bundle $E \to X$, we denote by $E^r|_{\F^r(X)}$ the restriction to $\F^r(X)$ of the product bundle $E^r$ on $X^r$. We will mostly be concerned with
\[
    (\Omega^1(\cL) - 0)^r |_{\F^r(X)} \quad \text{and} \quad (J^1(\cL) - 0)^r|_{\F^r(X)}
\]
which are respectively $(\bC^n - 0)^r$ and $(\bC^{n+1}-0)^r$ bundles over $\F^r(X)$. We will also make use of the evaluation map
\begin{align*}
    \ev \colon U_{\cC^0}(\cL) \times \F^r(X) &\lra (J^1(\cL) - 0)^r|_{\F^r(X)} \\
    (f, x_1, \dotsc, x_r) &\longmapsto (f(x_1), \dotsc, f(x_r)).
\end{align*}
It follows directly from the definitions that we have two pullback squares:
\[\begin{tikzcd}
\Z^r(\cL) \rar \dar & U(\cL) \times \F^r(X) \dar["j^1 \times \mathrm{id}"] \\
\Z^r_{\cC^0}(\cL) \rar \dar & U_{\cC^0}(\cL) \times \F^r(X) \dar["\ev"] \\
(\Omega^1(\cL) - 0)^r |_{\F^r(X)} \rar[hook] & (J^1(\cL) - 0)^r|_{\F^r(X)}
\end{tikzcd}\]
Over a point $\vec{v} = \left( (x_1,v_1),\dotsc,(x_r,v_r) \right) \in (J^1(\cL) - 0)^r|_{\F^r(X)}$, denote the fiber of the composition $\ev \circ (j^1 \times \mathrm{id})$ by
\[
    U(\vec{v}) \defeq \set{f \in U(\cL)}{j^1(f)(x_i) = (x_i, v_i) \text{ for each } i=1, \dotsc, r}
\]
and the fiber of $\ev$ by
\[
    U_{\cC^0}(\vec{v}) := \set{s \in \Gamma_{\cC^0}(J^1(\cL) - 0)}{s(x_i) = (x_i, v_i) \text{ for each } i = 1, \dotsc, r} \dpunct.
\]
We summarize the situation in the following commutative diagram:
\begin{equation}\label{eqn:main-proof-diagram}
\begin{tikzcd}
U(\vec{v}) \arrow[d, hook] \arrow[rr]                                                            &                           & U_{\cC^0}(\vec{v}) \arrow[d, hook]  \\
\Z^r(\cL) \arrow[rr, "j^1 \times \mathrm{id}"] \arrow[rd, "\ev \circ (j^1 \times \mathrm{id})"'] &                           & \Z^r_{\cC^0}(\cL) \arrow[ld, "\ev"] \\
             & (J^1(\cL) - 0)^r|_{\F^r(X)} &                                    
\end{tikzcd}
\end{equation}
Our strategy to prove the main theorem rests on the diagram~\eqref{eqn:main-proof-diagram}. 
Suppose for a moment that $\ev$ and $\ev \circ (j^1 \times \mathrm{id})$ were Serre fibrations. 
Then \cref{theorem:main-theorem} would follow by a Serre spectral sequence argument once we knew that the map between the fibers induces a homology isomorphism in a range of degrees. 
We will indeed establish such an isomorphism in \cref{sec:h-principle}, specifically \cref{h-principle-prescribed-derivatives}. 
However, although we will see shortly that $\ev$ is a fiber bundle (\cref{continuous-evaluation-is-fiber-bundle}), the map $\ev \circ (j^1 \times \mathrm{id})$ is only a Serre microfibration (\cref{jet-expansion-is-micro-fibration}). 
We recall this notion below, and explain why it is sufficient to carry out the outlined strategy.

\subsection{A micro review of microfibrations} \label{sec:microfibration}

We start by recalling the definition of a microfibration.

\begin{definition} \label{defn:microfibration}
A map $E \to B$ of topological spaces is a \emph{Serre microfibration} if for any $k \ge 0$, given a commutative diagram
\[\begin{tikzcd}
	D^k \times \{0\} \rar \dar & E \dar\\
	D^k \times [0, 1] \rar & B
\end{tikzcd}\]
there is an $\epsilon > 0$ and a map $D^k \times [0, \epsilon] \to E$ making the following diagram commute:
\[\begin{tikzcd}[baseline=(B.base)]
	D^k \times \{0\} \rar \dar & E \ar[dd]\\
	D^k \times [0, \epsilon] \ar[ru] \dar\\
	D^k \times [0, 1] \rar & |[alias=B]| B
\end{tikzcd}  \qedhere\]
\end{definition}

The following direct consequence of the definition is an abundant source of examples:
\begin{lemma}\label{lemma:open-in-fibration-is-microfibration}
If $p \from E \to B$ is a Serre fibration and $U \subset E$ is open, then $p|_U \from U \to B$ is a Serre microfibration.\qed
\end{lemma}

In this paper, wherever we use the terms fibration and microfibration, we mean Serre fibration and Serre microfibration respectively.
Contrary to the case of fibrations, the homotopy types of the fibers of a general microfibration can vary.
Nonetheless, comparing the total spaces of a microfibration and a fibration can done via a result going back to Weiss and generalized by Raptis.
\begin{theorem}[{\cites[Lemma~2.2]{weiss_what_2005}[Theorem~1.3]{raptis_serre_2017}}]\label{thm:raptismicro}
Let $p \colon E \to B$ be a Serre microfibration, $q \colon V \to B$ be a Serre fibration, and $f \colon E \to V$ a map over $B$. Suppose that $f_b \colon p^{-1}(b) \to q^{-1}(b)$ is $n$-connected for some $n \geq 1$ and for all $b \in B$. Then the map $f \colon E \to V$ is $n$-connected.
\end{theorem}

As we are interested in homology rather than homotopy groups, we will need to slightly adapt Raptis' theorem to our needs. We first introduce some notation.
\begin{definition}
For a map $p \colon E \to B$, its fiberwise (unreduced) $k$th suspension is defined to be
\[
	\Sigma^k_B E = \left( E \times [0,1] \times S^{k-1} \right) \bigg/ \big((e,0,s) \sim (e,0,s') \text{ and } (e,1,s) \sim (e',1,s) \text{ when } p(e) = p(e') \big).
\]
The fiber of the natural map $\Sigma^k_B p \from \Sigma^k_B E \to B$ induced by $p$ is the unreduced $k$th suspension of the fiber of $p$ (here modeled as the join with the sphere $S^{k-1}$):
\[
	(\Sigma^k_B p)^{-1}(b) = \Sigma^k p^{-1}(b) \qquad \forall b \in B. \qedhere
\]
\end{definition}

\begin{definition}
For a natural number $m$, say a map of topological spaces $A \to B$ is \emph{homology $m$-connected} if it induces an isomorphism on homology groups $H_i(A) \to H_i(B)$ for $i < m$ and a surjection when $i = m$.
\end{definition}

\begin{lemma}\label{adapted-microcomparison}
Let $q \colon V \to B$ be a fiber bundle, and $p \colon U \to B$ be the restriction of a fiber bundle $E \to B$ to an open subset $U \subset E$. Let $f \colon U \to V$ be a map over $B$ and suppose that for every $b \in B$, the restriction to the fiber
\[
	f_b \colon p^{-1}(b) \lra q^{-1}(b)
\]
is homology $m$-connected. Then $f \colon U \to V$ is homology $m$-connected.
\end{lemma}

\begin{proof}
For any $b \in B$, the suspension of $f$ on the fiber
\[
	\Sigma^2 f_b \colon \Sigma^2 p^{-1}(b) \lra \Sigma^2 q^{-1}(b)
\]
induces an isomorphism in homology in degrees $* \leq m+1$ and a surjective morphism in degree $* = m+2$. As both spaces are simply connected, the homology Whitehead theorem implies that this map is $(m+2)$-connected. We would like to apply \cref{thm:raptismicro} to $\Sigma^2_B f$, but $\Sigma^2_B U \subset \Sigma^2_B E$ is not open and it is unclear if $\Sigma^2_B U \to B$ is a microfibration. We resolve the issue by enlarging slightly the space to a homotopy equivalent one. More precisely, let
\[
	W = \big(\Sigma^2_B U\big) \cup \big( (E \times (0.5,1] \times S^1)) / \sim\big) \subset \Sigma^2_B E,
\]
and denote by $E_b, W_b, U_b$ the fibers of the respective spaces above a point $b \in B$. Using in each fiber the homotopy equivalence $\big((E_b \times (0.5,1] \times S^1)/\sim\big) \simeq S^1$ given by collapsing gives a homotopy equivalence
\[
	(W, W_b) \overset{\simeq}{\lra} (\Sigma^2_B U, \Sigma^2 U_b)
\]
for all $b \in B$. Now, the fiberwise suspension of the fiber bundles $E \to B$ and $V \to B$ are fiber bundles. As $W \subset \Sigma^2_B E$ is open, the restriction $W \to B$ is a microfibration. Applying \cref{thm:raptismicro} to the composite
\[
	W \overset{\simeq}{\lra} \Sigma^2_B U \overset{\Sigma^2_B f}{\lra} \Sigma^2_B V
\]
and using that the first map is a homotopy equivalence, we obtain that $\Sigma^2_B f \colon \Sigma^2_B U \to \Sigma^2_B V$ is $(m+2)$-connected. Hence it is homology $(m+2)$-connected. Comparing the Mayer--Vietoris sequences of the fiberwise suspensions finally shows that $f \colon U \to B$ is homology $m$-connected.
\end{proof}

\subsection{Finishing the proof of \texorpdfstring{\cref{maintheorem:comparison}}{Theorem~\ref*{maintheorem:comparison}}}

As promised, we show that the maps in the diagram~\eqref{eqn:main-proof-diagram} are respectively a microfibration and a fibration.

\begin{lemma}\label{continuous-evaluation-is-fiber-bundle}
The evaluation map
\begin{equation*}
\begin{split}
	\mathrm{ev} \colon \Gamma_{\cC^0}(J^1(\cL) - 0) \times \F^r(X) &\lra (J^1(\cL) - 0)^r|_{\F^r(X)} \\
	(f, x_1, \dotsc, x_r) &\longmapsto (f(x_1), \dotsc, f(x_r))
\end{split}
\end{equation*}
is a fiber bundle. Therefore, so is the pullback $\Z^r_{\cC^0}(\cL) \to (\Omega^1(\cL) - 0)^r |_{\F^r(X)}$.
\end{lemma}

\begin{proof}
We first treat the case $r = 1$ to lighten the notation. Let $(x,v) \in J^1(\cL) - 0$, with $x \in X$ and $0 \neq v \in J^1(\cL)_x$. Using charts on the manifold $X$, we choose a small open ball $B(x,1) \subset \bR^{2n} \subset X$ centered at $x$ and of radius $1$. Using the local triviality of the jet bundle, we obtain a homeomorphism $(J^1(\cL) -0)|_U \cong U \times (\bR^{2n+2} - 0)$. We choose a small $\varepsilon > 0$ and let $B(v,\varepsilon) \subset \bR^{2n+2} - 0$ be a small open ball neighborhood of $v$. Pick continuous maps
\[
	\varphi \colon B(v,\varepsilon/4) \lra \mathrm{Homeo}(B(v,\varepsilon))
\]
and
\[
	\phi \colon B(x,1/4) \lra \mathrm{Homeo}(B(x,1))
\]
such that $\varphi(w)$ is a homeomorphism sending $v$ to $w$ and is the identity outside $B(v,\varepsilon/2)$, and $\phi(y)$ is a homeomorphism sending $x$ to $y$ and is the identity outside $B(x,1/2)$. 
Slightly abusing notation, we still denote by $\varphi(w)$ and $\phi(y)$ the homeomorphisms of $J^1(\cL) - 0$ and $X$ respectively obtained by extending by the identity. We use them to construct a local trivialization of the evaluation map above the subset $A = B(x,1/4) \times B(v,\varepsilon/4) \subset J^1(\cL) - 0$ as follows:
\begin{align*}
	A \times \{s \in \Gamma_{\cC^0}(J^1(\cL) - 0) \mid s(x) = v \} &\overset{\cong}{\longleftrightarrow} \mathrm{ev}^{-1}(A) = \{ (y,s) \mid y \in B(x,1/4), \ s(y) \in B(v,\varepsilon/4) \} \\
	\big((y,w), s\big) &\longmapsto \big(y, \varphi(w) \circ s \circ \phi(y)^{-1}\big) \\
	\big(y, s(y), \varphi(s(y))^{-1} \circ s \circ \phi(y)\big) &\longmapsfrom \big(y,s\big).
\end{align*}
One directly checks that the two given maps are inverse to each other.

We now return to the general case where $r \geq 2$. To construct a local trivialization above a neighborhood of a point $\big( (x_1,v_1),\dotsc,(x_r,v_r) \big) \in (J^1(\cL) - 0)^r|_{\F^r(X)}$, it suffices to pick neighborhoods of the $x_i$ and apply the argument above at each of them. By choosing the neighborhoods small enough and disjoint, the homeomorphisms constructed as above may be composed to obtain a local trivialisation.
\end{proof}

On the algebraic side, for the jet evaluation map to even be surjective, we need the line bundle $\cL$ to have enough sections. As a direct consequence of the definition of jet ampleness (see \cref{defn:jet-ample}), we have the following refinement of \cite[Lemma~3.2]{vakil_discriminants_2015}.
\begin{lemma}\label{jet-expansion-is-micro-fibration}
Suppose that $\cL$ is $(2r-1)$-jet ample. Then the map
\begin{equation*}
\begin{split}
	\set{(f,x_1,\dotsc,x_r) \in \Gammahol(\cL) \times \F^r(X)}{ \forall i, \ f(x_i) = 0} &\lra (\Omega^1(\cL))^r |_{\F^r(X)} \\
	(f, x_1, \dotsc, x_r) &\longmapsto (\diff f(x_1), \dotsc, \diff f(x_r))
\end{split}
\end{equation*}
is a fiber bundle. The subset $\Z^r(\cL)$ is open in the domain, hence the restriction
\[
    \Z^r(\cL) \lra (\Omega^1(\cL) - 0)^r |_{\F^r(X)}
\]
is a microfibration.
\end{lemma}

\begin{proof}
The map of the lemma is the pullback of
\begin{align*}
    \Gammahol(\cL) \times \F^r(X) &\lra (J^1(\cL))^r|_{\F^r(X)} \\
    (f,x_1,\dotsc,x_r) &\longmapsto (j^1(f)(x_1),\dotsc,j^1(f)(x_r))
\end{align*}
along the inclusion $(\Omega^1(\cL))^r|_{\F^r(X)} \hookrightarrow (J^1(\cL))^r|_{\F^r(X)}$. It is then enough to show that this latter map is a fiber bundle. Both the domain and codomain are vector bundles on $\F^r(X)$ and the map is linear in each fiber. By the assumption on the jet ampleness of $\cL$ it is also fiberwise surjective, and is therefore an affine bundle. The second part of the lemma follows directly from \cref{lemma:open-in-fibration-is-microfibration}.
\end{proof}

\begin{proof}[Proof of \cref{theorem:main-theorem}]
It suffices to apply \cref{adapted-microcomparison} to the diagram~\eqref{eqn:main-proof-diagram}. Its assumptions are fulfilled by virtue of \cref{jet-expansion-is-micro-fibration,continuous-evaluation-is-fiber-bundle,h-principle-prescribed-derivatives}.
\end{proof}

\section{Stability with rational coefficients}
\label{sec:rational-model}

In this section, we construct a commutative differential graded algebra (\cdga{}) modelling the rational homotopy type of $\Z^r_{\cC^0}(\cL)$. Our construction only depends on known models for configuration spaces and mapping spaces, as well as basic methods from rational homotopy. We first recall the former, and use the latter to deduce \cref{maintheorem:stability}.

As a matter of notation, we will write $\Symgr(V)$ for the graded commutative algebra freely generated by a graded vector space $V$; this is sometimes denoted by $\Lambda(V)$, especially in the rational homotopy theory literature, but we will not use that notation.
If $V$ has a chosen (homogeneous) basis $x_1, \dots, x_k$, we will also write $\QQ[x_1, \dots, x_k]$ for $\Symgr(V)$ and $A[x_1, \dotsc, x_k]$ for $A \tensor \Symgr(V)$, where $A$ is any commutative graded algebra.
If $A$ is already a commutative \emph{differential} graded algebra (\cdga{}) and $y_1, \dotsc, y_n \in A$ are cocycles, then $(A[x_1, \dots, x_k], dx_1 = y_1, \dotsc, dx_k = y_k)$ will denote the \cdga{} whose differential extends the differential in $A$ by the added formulas.

We shall write $\APL(-)$ for the functor of piecewise linear differential forms, as constructed in \cite[Section~10]{felix_rational_2001}. We recall that a \cdga{} $(A,d)$ models a space $X$ if there is a zigzag of quasi-isomorphisms between $\APL(X)$ and $(A,d)$.

\subsection{Rational homotopy of configuration spaces}
Fulton and MacPherson, in \cite{fulton_compactification_1994}, first gave a rational model in the sense of Sullivan for the configuration spaces of points on a smooth projective complex variety.
This model was later improved by Kriz \cite{kriz_rational_1994} and Totaro \cite{totaro_configuration_1996} and we recall its construction here.
As before, let $X$ be smooth projective of complex dimension $n$.

\begin{construction}\label{construction:totaro-model}
Let $r \geq 1$ be a natural number. For integers $1 \leq a,b \leq r$, $a\neq b$, denote by $\pi_a \colon X^r \to X$ and $\pi_{ab} \colon X^r \to X^2$ the obvious projections. Let $C_r(X)$ be the quotient of the graded commutative algebra
\[
    H^*(X^r;\bQ)[G_{ab}] \dpunct,
\]
where the $G_{ab}$ are generators in degree $2n-1$ for $1 \leq a,b \leq r$, $a \neq b$, modulo the following relations:
\begin{align*}
    G_{ab} &= G_{ba} \\
    (G_{ab})^2 &= 0 \qquad \text{(automatic from graded commutativity)}\\
    G_{ab}G_{ac} + G_{bc}G_{ba} + G_{ca}G_{cb} &= 0 \qquad \text{for $a,b,c$ distinct} \\
    \pi_a^*(x)G_{ab} &= \pi_b^*(x)G_{ab} \qquad \text{for } a\neq b, x \in H^*(X;\bQ).
\end{align*}
Define a differential $d$ on $C_r(X)$ by
\[
    d(G_{ab}) = \pi_{ab}^*(\Delta)
\]
where $\Delta \in H^{2n}(X^2;\bQ)$ is the class of the diagonal. Then the \cdga{} $(C_r(X), d)$ is a rational model for the configuration space $\F^r(X)$.
\end{construction}

\subsection{Rational homotopy of section spaces}
\label{sec:recollections-sectionspace}

We recall the results from \cite[Section 8.1]{aumonier_h-principle_2022} concerning the rational cohomology of the continuous section space $\Gammacon(J^1(\cL) -0)$. The rational homotopy equivalences
\[
    \Gammacon(J^1(\cL) -0) \simeq_\bQ \map(X, K(\bQ,2n+1)) \simeq \prod_{i=0}^{2n+1} K(H^{2n+1-i}(X;\bQ), i)
\]
imply that
\[
    H^*(\Gammacon(J^1(\cL) -0); \bQ) \cong \Symgr\bigg( \bigoplus_i H^{2n+1-i}(X;\bQ)^\vee [-(2n+1)]\bigg)
\]
where $H^{*}(X;\bQ)^\vee = \Hom(H^{*}(X;\bQ), \bQ)$ denotes the dual vector space (in cohomological degree $-*$), and $[k]$ indicates a cohomological degree shift ($A[k]^i = A^{i+k}$). Using that homology is linearly dual to cohomology and Poincaré duality, we have isomorphisms:
\[
    H^{2n+1-*}(X)^\vee \cong H_{2n+1-*}(X) \cong H^{*-1}(X)
\]
which can be used to rewrite
\begin{equation}\label{eqn:stablecohomology-free-gca}
    H^*(\Gammacon(J^1(\cL) -0); \bQ) \cong \Symgr\bigg( H^*(X;\bQ) [-1]\bigg).    
\end{equation}
We will need to understand the morphism induced in cohomology by the evaluation map
\[
    \mathrm{ev} \colon \map(X, K(\bQ,2n+1)) \times X \lra K(\bQ,2n+1).
\]
This is explained by Haefliger in \cite[Section~1.2]{haefliger_rational_1982} and we transcribe here his words in our notation. Let $\{b_j\}$ be a homogeneous basis of the graded vector space $H^*(X;\bQ)$.
Let $\{b_j^\vee\}$ be the dual basis under the Poincaré pairing, so that $\abs*{b_j^\vee} = 2n - \abs*{b_j}$ and $b_i \smile b_j^\vee = \delta_{ij} [X]$ if $\abs*{b_i} = \abs*{b_j}$. Here $\delta_{ij}$ is the Kronecker delta, and $[X] \in H^{2n}(X;\bQ)$ denotes the fundamental cohomology class of $X$.
Under the isomorphism~\eqref{eqn:stablecohomology-free-gca}, each $b_j$ corresponds to a class of degree $\abs*{b_j} + 1$ in the cohomology of $\Gammacon(J^1(\cL) -0)$ which we denote by $sb_j$ (the shifted class).
\begin{lemma}\label{lemma:pullback-generator-evaluation}
Let $\chi \in H^{2n+1}(K(\bQ,2n+1); \bQ)$ be the canonical generator. The morphism induced in cohomology by the evaluation map
\[
    \Gammacon(J^1(\cL) -0) \times X \simeq_\bQ \map(X, K(\bQ,2n+1)) \times X \tolabel{\ev} K(\bQ,2n+1)
\]
sends $\chi$ to
\[
    \sum_j sb_{j} \otimes b_j^\vee \in H^*(\Gammacon(J^1(\cL) -0); \bQ) \otimes H^*(X;\bQ).
\]
\end{lemma}
\begin{proof}
To avoid cluttering the argument with too much notation, we first explain the general case of a map $f \colon Z \times X \to K(\bQ,2n+1)$ by closely following Haefliger's argument. (We will later take $f$ to be the evaluation map.) By adjunction, $f$ is the same datum as a map $g \colon Z \to \map(X, K(\bQ,2n+1))$. Let $g_i$ be the composition of $g$ with the projection onto the $i$-th factor:
\[
    Z \lra \map(X,K(\bQ,2n+1)) \simeq \prod_{j=0}^{2n+1} K(H^{2n+1-j}(X;\bQ), j) \lra K(H^{2n+1-i}(X;\bQ), i).
\]
Haefliger then explains that the morphism induced in cohomology is given by
\begin{align*}
    g_i^* \colon H^i( K(H^{2n+1-i}(X;\bQ), i); \bQ) \cong H^{2n+1-i}(X;\bQ)^\vee &\lra H^i(Z;\bQ) \\
    a &\longmapsto a \cap f^*(\chi)
\end{align*}
where $a \cap (z \otimes x) = a(x) z$ for $a \in H^*(X)^\vee$, $x \in H^*(X)$ and $z \in H^*(Z)$. We now take $f$ to be the evaluation map, hence $g$ to be the identity. Decomposing in the chosen bases, we may a priori write
\[
    \mathrm{ev}^*(\chi) = \sum_{\abs*{b_j} = \abs*{b_k}} \lambda_{jk} \cdot (sb_j) \otimes b_k^\vee
\]
for some constants $\lambda_{jk} \in \bQ$. Using that $g$ is the identity, we have for all $a \in H^{2n + 1 - i}(X)^\vee \cong H^{i - 1}(X)$:
\[
    a \cap \mathrm{ev}^*(\chi) = sa \dpunct.
\]
Varying $a$ through $\{b_j\}$ finishes the proof.
\end{proof}

\subsection{A rational model for the continuous section space with marked points}

We are now ready to construct a \cdga{} which, we will shortly show, computes the rational cohomology of the space $\Z^r_{\cC^0}(\cL)$.
\begin{construction}\label{rational-model}
Let $r \geq 1$ be a natural number. We define $A_r(\cL)$ to be the commutative graded algebra
\begin{equation}
    A_r(\cL) = C_r(X) \tensor \Symgr\big( H^*(X; \QQ)[-1] \dsum \QQ\gen*{\alpha_1, \dotsc, \alpha_r} \dsum \QQ\gen*{\eta_1, \dotsc, \eta_r}\big) \dpunct,
\end{equation}
where each $\alpha_i$ is in degree $2n - 1$, each $\eta_i$ is in degree $2n$, and $C_r(X)$ is the rational model of $\F^r(X)$ recalled in \cref{construction:totaro-model}.
Corresponding to the projections $\pi_i \from \F^r(X) \to X$, let $\pi_i^* \from H^*(X) \to H^*(X)^{\tensor r} \to C_r(X)$ be the $i$th inclusion. 
Define
\[
    \epsilon_i \defeq \sum_j \pi_i^*(b_j^\vee) \otimes sb_j \in C_r(X) \otimes \Symgr\big(H^{*}(X; \QQ)[-1] \big).
\]
which is an element in degree $2n + 1$. In fact, by \cref{lemma:pullback-generator-evaluation}, it is the class represented by the composite 
\[
    \Gammacon(J^1(\cL) -0) \times \F^r(X) \tolabel{\id \times \pi_i} \Gammacon(J^1(\cL) -0) \times X \simeq_\bQ \map(X, K(\bQ,2n+1)) \times X \tolabel{\ev} K(\bQ,2n+1).
\]
We define a differential on $A_r(\cL)$ by the tensor product of the differential on $C_r(X)$ and the differential on the second tensor factor given by 
\begin{align*}
    d(\alpha_i) &= \pi_i^*(e(\Omega^1(\cL))) \dpunct,\\
    d(\eta_i) &= \epsilon_i - \pi_i^*(c_1(\cL))\alpha_i
\end{align*}
and sending the other generators to $0$.
\end{construction}

\begin{remark}
We can define a $\fS_r$ action on $A_r(\cL)$ by acting on $C_r(X)$ by permuting coordinates of $X^r$, trivially on $H^{*}(X)[-1]$ and by permuting the $\alpha_i$ and $\eta_i$. It is clear that the differential defined above is $\fS_r$-equivariant.
\end{remark}

We will see below that for $\cL$ sufficiently ample, the \cdga{} $A_r(\cL)$ varies in a fixed isomorphism class if $c_1(\cL)$ is multiplied by a scalar (eg if $\cL$ is replaced by $\cL^{\tensor d}$ for large $d$).
In cases relevant to our results this class in $H^2(X)$ will lie in the \emph{ample cone} but the construction of $A_r(\cL)$ does not depend on this property.
Therefore we make the following definition:

\begin{construction}\label{stable-cdga}
Fix $c \in H^2(X)$ and $A_r(X, c)$ be the \cdga{} that as a graded algebra is the same as \cref{rational-model}:
\[A_r(X, c) = C_r(X) \tensor \Symgr\big( H^*(X; \QQ)[-1] \dsum \QQ\gen*{\alpha_1, \dotsc, \alpha_r, \eta_1, \dotsc, \eta_r}\big)\]
but with differential given by
\begin{align*}
d \alpha_i &= \pi_i^*([X]) \dpunct,\\
d \eta_i &= \epsilon_i - c \alpha_i
\end{align*}
and vanishing on the generators in $H^{*}(X)[-1]$, with the same notation as in \cref{rational-model}.
As shown in \cref{rational-model-stabilizes}, the \cdga{} $A_r(X, c)$ remains in the same isomorphism class if $c$ is replaced by a non-zero multiple, so we will allow $c \in \PP H^2$ with the same notation if we only care about $A_r(X, c)$ up to isomorphism, which will often be the case.
\end{construction}

\begin{remark}
Tracing through the definitions, it is not hard to verify that $A_r(X, c)$ only depends on $r$, the Poincaré duality algebra $H = H^*(X)$ and $c \in H^2$ (or $c \in \PP H^2$, up to isomorphism). 
Therefore one could define an analogous \cdga{} $A_r(H, c)$ in this broader generality.
We shall not do this.
\end{remark}

\begin{theorem}\label{serre-ss-degeneration}
The commutative differential graded algebra $A_r(\cL)$ of \cref{rational-model} is a rational model of $\Z^r_{\cC^0}(\cL)$. In particular, there is an $\fS_r$-equivariant isomorphism
\[ H^*(\Z^r_{\cC^0}(\cL); \bQ) \isom H^*(A_r(\cL)) \dpunct.\]
\end{theorem}

We will prove this theorem in the next section. But first, we collect a few computational lemmas and direct consequences.

\begin{proposition}\label{euler-class-non-vanishing}
Let $\cL$ be an ample line bundle. Then there exists a $d_0 \geq 1$ such that $e(\Omega^1(\cL^d)) \neq 0$ for all $d \geq d_0$.
\end{proposition}
\begin{proof}
For an integer $d$, we compute the Euler class:
\[
    e(\Omega^1(\cL^d)) = c_n(\Omega^1(\cL^d)) = c_n(\Omega^1_X \tensor \cL^d) = \sum_{i=0}^n c_i(\Omega^1_X) c_1(\cL^d)^{n-i} = \sum_{i=0}^n c_i(\Omega^1_X) c_1(\cL)^{n-i} d^{n-i}.
\]
Recall, eg from the Nakai--Moishezon criterion, that ampleness of $\cL$ implies that $c_1(\cL)^n[X] > 0$. In particular
\[
    e(\Omega^1(\cL^d))[X] = (c_1(\cL)^n[X])d^n + o(d^n)
\]
is a polynomial in $d$ of degree $n$. Thus, there exists a $d_0\geq 1$ such that this polynomial does not vanish when evaluated at all $d \geq d_0$.
\end{proof}

\begin{remark}\label{euler-class-bound}
The rational root theorem implies that $d_0 = 1+ \abs*{\chi(X)}$ suffices in the proposition.
For curves the polynomial is $a d - \chi(X)$ with $a = c_1(\cL)[X] \ge 1$, so $d_0 = 3$ is sufficient.
We are not aware of a bound that is uniform in all $X$ of a given dimension $n$ for general $n > 1$.
\end{remark}

We can now state and prove our main stability result:

\begin{proposition}\label{rational-model-stabilizes}
Let $\cL$ be an ample line bundle, and let $d_0$ be as in \cref{euler-class-non-vanishing}. Then for all $d \geq d_0$, the \cdga{} $A_r(\cL^d)$ is isomorphic to $A_r(X, c_1(\cL))$.
Moreover, the \cdga{} $A_r(X, c)$ does not depend, up to isomorphism, on the choice of representative $c$ for $[c] \in \PP H^2(X)$.
\end{proposition}
\begin{proof}
The Euler class $e(\Omega^1(\cL^d))$ lives in $H^{2n}(X;\bZ) \cong \bZ$. Let us write $m(d) = e(\Omega^1(\cL^d))[X] \in \bZ$ for that number. By assumption, we have $m(d) \neq 0$ for all $d \geq d_0$. 
Fix $c \in H^2 X$ such that $c_1(\cL) = \lambda c$ for some $\lambda \in \QQ^\times$.
We construct an explicit morphism:
\[
A_r(\cL^{d}) \lra A_r(X, c)
\]
given by the identity on the $C_r(X)$ tensor factor and sending the generators according to:
\[
    sb_j \mapsto \lambda^{-1} m(d) \cdot sb_j, \quad \alpha_i \mapsto m(d) \cdot  \alpha_i, \quad \eta_i \mapsto \lambda^{-1}m(d) \cdot  \eta_i \dpunct.
\]
One directly checks that this defines a morphism of \cdga{}s. It is visibly an isomorphism when $m(d) \ne 0$.
\end{proof}

When $r = 1$, we can compute $H^*(A_1(X, c))$ exactly and it is independent of $c$, as in \cref{one-point-stability}.
But for this computation we need a couple of elementary lemmas about \cdga{}s, which we state and prove here for the reader's convenience.
The parity assumptions below are not crucial but are the only cases we need and they simplify both the statements and the proofs.

\begin{lemma}\label{lemma:homological-algebra}
Consider the \cdga{} $(A[x], dx = z)$, where $A$ is a \cdga{}, $x$ has (positive) even degree and $z \in A$ is a cocycle.
Suppose that multiplication by $z$ is ``exact'' on $A$, ie $zA = \set{a \in A}{az = 0}$. 
Then the map $A[x] \to A/zA$ taking $x \mapsto 0$ is a quasi-isomorphism.
\end{lemma}
\begin{proof}
Let $|x| = i$ and consider $A[x]$ as a double complex 
\[(A[x])^{p, q} = A^{p + (i+1)q} x^{-q}\]
(supported on $q \le 0$) with $p$-differential $d_A \colon A^{p + (i + 1)q}x^{-q} \to A^{p + 1 + (i + 1)q} x^{-q}$ and $q$-differential
\[d_z^{p,q} \colon A^{p + (i + 1)q} x^{-q} \to A^{p + (i + 1)q + (i + 1)} x^{-q - 1} = A^{p + (i + 1)(q + 1)}x^{-(q+1)}\]
defined by $d_z^{p,q}(a x^{-q}) = (-1)^{p + q +1}q \cdot az x^{-q - 1}$.
Then we have the spectral sequence
\[E_1^{p,q} = H^q(A[x]^{p, \bullet}, d_z) \implies H^{p + q}(A[x]).\]
But at any $q < 0$, the differential $d_z^{p, q}$ is, up to a non-zero scalar, multiplication by $z$.
So by the assumption on $z$ we have the exactness $\ker d_z^{p, q} = \mathrm{Im} \ d_z^{p, q - 1}$ for $q < 0$ and arbitrary $p$.
Therefore $E_1$ is supported on $q = 0$ with
\[E_1^{p, 0} = A^p/z A^p\]
and differential $d_1$ induced by $d_A$.
So the spectral sequence collapses on $E_2$ and we get the claimed isomorphism.
\end{proof}

In practice and in all our applications, to verify the assumptions of \cref{lemma:homological-algebra} above it is useful to use the following result:

\begin{lemma}\label{lemma:exact-multiplication}
Suppose $A$ is of the form $B[t]$ with $|t| = i$ odd and $z \in \lambda t + B^i$ for some $\lambda \in \bQ^\times$. Then $zA = \set{a \in A}{az = 0}$.
\end{lemma}
\begin{proof}
Reduce to the case $\lambda = 1$ without loss of generality, so $z = t + b$ for some $b \in B^i$.
Since $|z| = i$ is odd, $z^2 = 0$, so $a \in zA \implies az = 0$.
So suppose $za = 0$, and assume $a$ is homogeneous of degree $j$.
Then $a = b_1 t + b_0$ for some $b_0 \in B^j$, $b_1 \in B^{j-i}$.
Now 
\[0 = az = (b_1 t + b_0)(t + b) = b_1 t b + b_0 t + b_0 b = (b_0 - b_1b)t + b_0 b \implies b_0 = b_1b ,\]
since $t^2 = 0$ and $t b = (-1)^{i^2} bt = - bt$.
But now $a = (b_1t + b_1b) = b_1 z \in zA$.
\end{proof}

In combination, the two lemmas say that a \cdga{} $(B[t, x], dt = \dots, dx = t + \dots)$ is quasi-isomorphic to $B$ by the map sending $x \mapsto 0, dx \mapsto 0$.
As promised, we can now compute $H^*(A_1(X, c))$.

\begin{proof}[Proof of \cref{one-point-stability}]\label{one-point-stability-proof}
Note that since $X$ is a compact complex manifold, its reduced cohomology $\widetilde{H}^*(X)$ is Poincaré dual to $H^{* < 2n}(X)$, we will freely use this identification. 
Denote the fundamental cohomology class of $X$ by $[X]$, so that $H^*(X) \isom H^{* < 2n}(X) \dsum \QQ\gen*{[X]}$ as graded vector spaces.

For $r = 1$ and $c = c_1(\cL)$, the \cdga{} $A_1(X, c)$ has the form 
\[A_1(X, c) = H^*(X) \tensor \Symgr\big(H^{*}(X)[-1] \dsum \QQ\gen*{\alpha, \eta}\big) \]
where $|\alpha| = 2n-1$, $|\eta| = 2n$ (and setting $n = \dim_\CC X$) and differential defined by
\[d\alpha = c_n(\Omega^1(\cL)) \dpunct; \qquad d\eta = s[X] + \epsilon - c \alpha \dpunct.\]
with $[X] \in H^{2n}(X)$ the fundamental cohomology class, $s[X] = 1 \tensor s[X]$ its corresponding generator in $\Symgr(H^{*}(X)[-1])$ and $\epsilon$ a class in $\widetilde{H}^*(X) \tensor \Symgr(H^{* < 2n}(X)[-1])$, see \cref{stable-cdga} for details.

Now, by \cref{lemma:homological-algebra,lemma:exact-multiplication}, 
\[A_1(X, c) \simeq H^*(X) \tensor \Symgr(H^{* < 2n}(X)[-1] \dsum \QQ\gen*{\alpha}) \isom (H^*(X)[\alpha] ) \tensor \Symgr(H^{* < 2n}(X)[-1]) \dpunct,\]
by the map $\eta \mapsto 0$, $s[X] \mapsto c\alpha - \epsilon$.
Note that the final factorization is as \cdga{}s (ie its second factor has vanishing differential) and its first factor $(H^*(X)[\alpha], d\alpha = c_n(\Omega^1(\cL))\,)$ is a model of $\Omega^1(\cL) - 0$.
The claim follows by taking cohomology.
\end{proof}

\subsection{Proof of \texorpdfstring{\cref{serre-ss-degeneration}}{the rational model}}
\label{rational-model-proof}

The proof of \cref{serre-ss-degeneration} relies on recognizing $\Z^r_{\cC^0}(\cL)$ as the following pullback: 
\begin{equation}\label{main-pullback}
\begin{tikzcd}
\Z_{\cC^0}^r(\cL) \arrow[d] \arrow[r]                                & \F^r(X) \times \Gammacon(J^1(\cL) -0) \arrow[d, "\mathrm{ev}"] \\
(\Omega^1(\cL) - 0)^r|_{\F^r(X)} \arrow[r, hook] & (J^1(\cL) - 0)^r|_{\F^r(X)}                       
\end{tikzcd}
\end{equation}
which is also a homotopy pullback, as \cref{continuous-evaluation-is-fiber-bundle} shows that the right-hand vertical arrow is a fibration. Our strategy is then to apply the Eilenberg--Moore theorem:
\begin{theorem}[Eilenberg--Moore]\label{eilenberg-moore-theorem}
Let\nopagebreak
\[%
\begin{tikzcd}
E \times_B E' \arrow[d] \arrow[r] & E \arrow[d] \\
E' \arrow[r]                      & B          
\end{tikzcd}
\]
be a pullback square of connected spaces, where $E \to B$ is a fibration with connected fiber $F$. Suppose that the cohomology groups $H^i(F;\bQ)$ are finite dimensional and that the monodromy action of $\pi_1(B)$ on them is trivial. Then the canonical morphism
\[
    \APL(E) \otimes^\bL_{\APL(B)} \APL(E') \lra \APL(E \times_B E')
\]
is a quasi-isomorphism.
\end{theorem}
\begin{proof}
For a proof see \cite[Theorem~7.14]{mccleary_users_2001}.
\end{proof}

In order to explicitly compute the derived tensor product appearing in \cref{eilenberg-moore-theorem}, we will need to model the bottom arrow of \cref{main-pullback} as a cofibration between \cdga{}s. More precisely, we will model it as a relative Sullivan algebra \cite[Section~14]{felix_rational_2001}, which we shall derive from its Moore--Postnikov tower.
We first deal with the case where $r = 1$, and will explain later how to generalize by taking products.

Let us introduce some notation. Write
\[
    \iota \colon \Omega^1(\cL) - 0 \hookrightarrow J^1(\cL) - 0
\]
for the inclusion. We define the fiber product of the two bundles above $X$ as the pullback:
\begin{equation}\label{pullback-fibrewise-product}
\begin{tikzcd}
(\Omega^1(\cL) -0) \times_X (J^1(\cL) - 0) \arrow[d, "p_2"'] \arrow[r, "p_1"] & J^1(\cL) -0 \arrow[d, "q_1"] \\
\Omega^1(\cL) -0 \arrow[r, "q_2"']     & X           
\end{tikzcd}
\end{equation}
with $p_i$ and $q_i$ the projections. Notice that we also have a map
\[
    (\mathrm{id}, \iota) \colon \Omega^1(\cL) -0 \hookrightarrow (\Omega^1(\cL) -0) \times_X (J^1(\cL) - 0)
\]
given by the identity on the first factor and the inclusion on the second. Writing $a \in H^{2n-1}(\bC^n - 0)$ for the generator, the Serre spectral sequence of the bundle $q_2 \colon \Omega^1(\cL) - 0 \to X$ shows that
\[
    H^{2n+1}(\Omega^1(\cL) - 0;\bQ) \cong H^2(X;\bQ) \otimes \bQ a.
\]
Although $a$ does not survive in the spectral sequence if $e(\Omega^1(\cL)) \neq 0$, we will write 
\[
    x \cdot a \in H^{2n+1}(\Omega^1(\cL) - 0;\bQ)
\]
for $x \in H^2(X)$ using the isomorphism above. We will also write 
\[
    b \in H^{2n+1}(\bC^{n+1}-0) = H^{2n+1}(J^1(\cL)-0)
\]
for the generator. We will need the following computation:
\begin{lemma}\label{lemma:pullback-generator}
In the cohomology group $H^{2n+1}(\Omega^1(\cL) - 0)$, we have the equality $\iota^*(b) = c_1(\cL) \cdot a$.
\end{lemma}
\begin{proof}
The integration along the fibers of $q_2$, or Gysin map, gives an isomorphism
\[
    (q_2)_! \colon H^{2n+1}(\Omega^1(\cL) -0) \isomto H^2(X)
\]
such that $(q_2)_!(x \cdot a) = x$ in our notation. It thus suffices to check that $(q_2)_!(\iota^*(b)) = c_1(\cL)$. This follows by functoriality of the Gysin maps and the push-pull formula:
\begin{align*}
    (q_2)_!(\iota^*(b)) &= (q_1)_! \circ \iota_!(\iota^*(b)) \\
        &= (q_1)_! ( b \cup \iota_!(1)) \\
        &= (q_1)_! ( b \cup (q_1)^*(c_1(\cL))) \\
        &= c_1(\cL),
\end{align*}
where we have used the standard fact that $\iota_!(1)$ is the first Chern class of the normal line bundle of the inclusion $\iota \from \Omega^1(\cL) -0 \subto J^1(\cL)-0$.
\end{proof}

The following two lemmas are the main steps towards finding a relative Sullivan model for $\iota$.
\begin{lemma}\label{lemma:moore-postnikov-tower-k1}
Let
\[
    \big((\Omega^1(\cL) -0) \times_X (J^1(\cL) - 0) \big)_{f\bQ} \lra J^1(\cL) - 0
\]
be the fiberwise rationalization of $p_1$ (see eg \cite[Section~3]{bousfield_localization_1971}). It is a principal $K(\bQ,2n-1)$-fibration classified by a map
\[
    J^1(\cL) -0 \tolabel{k_1} K(\bQ,2n)
\]
whose homotopy class is represented by the cohomology class
\[
    k_1 = q_1^*( e(\Omega^1(\cL)) ) \in H^{2n}(J^1(\cL) - 0; \bQ) \dpunct.
\]
\end{lemma}
\begin{proof}
We have two commutative squares:
\[\begin{tikzcd}
(\Omega^1(\cL) -0) \times_X (J^1(\cL)-0) \arrow[d, "p_1"] \arrow[r] & \Omega^1(\cL) -0 \arrow[d, "q_2"] \arrow[r] & * \arrow[d] \\
J^1(\cL)-0 \arrow[r, "q_1"']                                                & X \arrow[r, "e(\Omega^1(\cL))"']    & {K(\bQ,2n)}
\end{tikzcd}\]
Notice first that both $p_1$ and $q_2$ are canonically oriented sphere bundles using the complex structure. In particular, they are nilpotent fibrations in the sense that the monodromy actions on the cohomology of their fibers are trivial, and it therefore makes sense to talk about fiberwise rationalization.
Now, the square on the left-hand is a pullback by definition, and the square on the right-hand is a pullback after fiberwise rationalization. So the outer square is a pullback after fiberwise rationalization.
\end{proof}

\begin{lemma}\label{lemma:moore-postnikov-tower-k2}
There is a commutative diagram\nopagebreak
\[\begin{tikzcd}
    \Omega^1(\cL) -0 \arrow[rd, "{(\mathrm{id}, \iota)}"'] \arrow[r, "{H^*(-;\bQ)\text{-iso}}"] & P \arrow[d] \arrow[r]   & * \arrow[d]           \\
    & \big((\Omega^1(\cL) -0) \times_X (J^1(\cL) - 0) \big)_{f\bQ} \arrow[r, "k_2"] & {K(\bQ,2n+1)}        
\end{tikzcd}\]
where $P$ is the pullback of the universal $K(\bQ,2n)$-fibration along the map $k_2$, whose homotopy class is represented by the cohomology class
\[
    k_2 = p_1^*(b) - p_2^*(c_1(\cL) \cdot a) \in H^{2n+1}((\Omega^1(\cL) -0) \times_X (J^1(\cL) - 0); \bQ),
\]
and the map
\[
    \Omega^1(\cL) -0 \lra P
\]
induces an isomorphism on rational cohomology.
\end{lemma}

\begin{proof}
First, observe that the composition
\[
    \Omega^1(\cL) - 0 \tolabel{(\mathrm{id}, \iota)} \big((\Omega^1(\cL) -0) \times_X (J^1(\cL) - 0)\big)_{f\bQ} \tolabel{k_2} K(\bQ,2n+1)
\]
is null-homotopic. Indeed, this is the content of \cref{lemma:pullback-generator}. Thus, if we consider the homotopy pullback square:
\[
\begin{tikzcd}
P \arrow[d] \arrow[r]                                                                 & * \arrow[d]   \\
\big((\Omega^1(\cL) -0) \times_X (J^1(\cL) - 0) \big)_{f\bQ} \arrow[r, "k_2"'] & {K(\bQ,2n+1)}
\end{tikzcd}\]
we obtain a map $\Omega^1(\cL) - 0 \to P$ by universal property, which makes the triangle of the lemma commute. We now have to verify that it is a rational cohomology equivalence. By the Eilenberg--Moore theorem applied to the principal fibration of \cref{lemma:moore-postnikov-tower-k1} classified by $k_1$, the rational cohomology of $\big((\Omega^1(\cL) -0) \times_X (J^1(\cL) - 0)\big)_{f\bQ}$ is given by that of the \cdga{}
\[
    H^*(X)[x_{2n+1}, y_{2n-1}], \quad d(x) = 0, \ d(y) = e(\Omega^1(\cL)).
\]
Here the indices on the variables indicate the cohomological degree. We have also used that $X$ is formal to use its cohomology as a model. 
Again using the Eilenberg--Moore theorem, we obtain a \cdga{} computing the rational cohomology of $P$ of the form
\[
    H^*(X)[x_{2n+1}, y_{2n-1}, z_{2n}], \quad d(x) = 0, \ d(y) = e(\Omega^1(\cL)), \ d(z) = x - c_1(\cL)y.
\]
By \cref{lemma:homological-algebra} (in the notation of that lemma, $z$ is $x-c_1(\cL)$ for us and the assumption is verified using \cref{lemma:exact-multiplication}), the cohomology of this \cdga{} is isomorphic to the cohomology of the following \cdga{}:
\[
    H^*(X)[x_{2n+1}, y_{2n-1}] \big/ \big(x - c_1(\cL)\big), \quad d(x) = 0, \ d(y) = e(\Omega^1(\cL)),
\]
which is also isomorphic to
\[
    H^*(X)[y_{2n-1}], \quad d(y) = e(\Omega^1(\cL)),
\]
and whose cohomology is $H^*(\Omega^1(\cL)-0;\bQ)$.
The morphism induced in rational cohomology by the map
\[
    (\mathrm{id}, \iota) \colon \Omega^1(\cL) -0 \lra \big((\Omega^1(\cL) -0) \times_X (J^1(\cL) - 0)\big)_{f\bQ}
\]
is surjective. Indeed this can be checked using the Eilenberg--Moore theorem applied to the pullback square~\eqref{pullback-fibrewise-product}. Therefore, using commutativity of the triangle
\[
\begin{tikzcd}
\Omega^1(\cL) - 0 \arrow[r] \arrow[rd] & P \arrow[d]                                                         \\
                                                & \big((\Omega^1(\cL) -0) \times_X (J^1(\cL) - 0) \big)_{f\bQ}
\end{tikzcd}
\]
we see that the morphism on cohomology
\[
    H^*(P;\bQ) \lra H^*(\Omega^1(\cL)-0;\bQ)
\]
is surjective. As shown above, both sides of this morphism are abstractly isomorphic rational vector spaces of finite dimension. Hence the morphism must be an isomorphism.
\end{proof}

\begin{remark}
If we furthermore assume that $X$ is nilpotent, then so are $\Omega^1(\cL) - 0$ and the space $P$ from \cref{lemma:moore-postnikov-tower-k2}. In that case, the map $\Omega^1(\cL) -0 \to P$ is a rational homotopy equivalence, being a rational cohomology equivalence between nilpotent spaces. In particular, the cohomology classes $k_1$ and $k_2$ are the k-invariants of the Moore--Postnikov tower of the rationalization of $\iota$ (which is well-defined as a map between nilpotent spaces). However, as we are ultimately interested in making a statement about cohomology, we do not have to require $X$ be nilpotent for our purposes.
\end{remark}

We are now ready to give a rational model of $\iota$ in the form of a cofibration of \cdga{}s (see \cite[Lemma~14.4]{felix_rational_2001} for a proof that relative Sullivan algebras are cofibrations in the commonly used projective model structure):
\begin{lemma}\label{lemma:model-inclusion}
Let $\iota$ be the inclusion
\[
    \iota \colon (\Omega^1(\cL) - 0)^r|_{\F^r(X)} \lra (J^1(\cL) - 0)^r|_{\F^r(X)} \dpunct.
\]
Choose representatives $e, c \in \APL(X)$ of the cohomology classes $e(\Omega^1(\cL))$ and $c_1(\cL)$.
Then the morphism $\APL(\iota)$ induced between the piecewise linear differential forms fits into a commutative diagram
\[
\begin{tikzcd}
\APL((J^1(\cL) - 0)^r|_{\F^r(X)}) \arrow[rr, "\APL(\iota)"]                           &  & \APL((\Omega^1(\cL) - 0)^r|_{\F^r(X)})                                                    \\
\APL(\F^r(X))[x_1, \dotsc, x_r] \arrow[rr, hook] \arrow[u, "\simeq"] &  & \APL(\F^r(X))[x_1,\dotsc,x_r,y_1,\dotsc,y_r, z_1,\dotsc,z_r] \arrow[u, "\simeq"']
\end{tikzcd}\]
where the added generators have degrees $|x_i| = 2n+1$, $|y_i|=2n-1$ and $|z_i|=2n$, and the differentials are given by $d(x_i) = 0$, $d(y_i) = \APL(\pi_i)(e)$ and $d(z_i) = x_i - \APL(\pi_i)(c)y_i$.
\end{lemma}
\begin{proof}
Rationally, the bundle $(J^1(\cL) - 0)^r|_{\F^r(X)}$ is trivial. In other words, we have an isomorphism
\[
    H^*((J^1(\cL) - 0)^r|_{\F^r(X)}; \bQ) \cong H^*(\F^r(X); \bQ) \otimes H^*((S^{2n+1})^r; \bQ) \dpunct.
\]
We thus obtain a quasi-isomorphism
\[
    \APL(\F^r(X))[x_1,\dotsc,x_r] \equivto \APL((J^1(\cL) - 0)^r|_{\F^r(X)})
\]
by sending each $x_i$ to a cocycle representing the generator of the $i$th sphere. Now, when $r = 1$, the result follows from applying the Eilenberg--Moore theorem twice using \cref{lemma:moore-postnikov-tower-k1,lemma:moore-postnikov-tower-k2}. The general case $r \geq 1$ follows by taking the $r$-fold product of the case $r=1$.
\end{proof}

It remains to prove our main theorem:
\begin{proof}[Proof of \cref{serre-ss-degeneration}]
Recall that we have a homotopy pullback diagram
\[
\begin{tikzcd}
\Z_{\cC^0}^r(\cL) \arrow[d] \arrow[r]                    & \F^r(X) \times \Gammacon(J^1(\cL) -0) \arrow[d, "\mathrm{ev}"] \\
(\Omega^1(\cL) - 0)^r|_{\F^r(X)} \arrow[r, hook] & (J^1(\cL) - 0)^r|_{\F^r(X)}                                   
\end{tikzcd}
\]
The Eilenberg--Moore theorem states that the induced morphism from the derived tensor product
\[
    \APL((\Omega^1(\cL) - 0)^r|_{\F^r(X)}) \otimes^\bL_{\APL((J^1(\cL) - 0)^r|_{\F^r(X)})} \APL(\F^r(X) \times \Gammacon(J^1(\cL) -0)) \equivto \APL(\Z_{\cC^0}^r(\cL))
\]
is a quasi-isomorphism. We thus have to show that this derived tensor product is quasi-isomorphic to the proposed \cdga{} of \cref{serre-ss-degeneration}. By \cref{lemma:model-inclusion}, we have a commutative diagram:
\[\begin{tikzcd}[column sep=small, row sep=large]
&[-35pt] \APL(\F^r(X) \times \Gammacon(J^1(\cL) -0)) \rar[equal] & \APL(\F^r(X) \times \Gammacon(J^1(\cL) -0)) \\
\APL((\Omega^1(\cL) - 0)^r|_{\F^r(X)}) & \APL((J^1(\cL) - 0)^r|_{\F^r(X)}) \arrow[l, "\APL(\iota)"'] \arrow[u, "\APL(\mathrm{ev})"]  \\
\APL(\F^r(X))[x_1,\dotsc,x_r,y_1,\dotsc,y_r, z_1,\dotsc,z_r] \arrow[u, "\simeq"'] & & \APL(\F^r(X))[x_1,\dotsc,x_r] \arrow[ul, "\simeq"', start anchor=north west, end anchor=south east] \arrow[uu, "\phi"] \arrow[ll, hook']
\end{tikzcd}\]
where $\phi$ is defined as the composition so that the diagram commutes. 
The derived tensor product is the homotopy pushout of the inner span, which is quasi-isomorphic to that of the outer span.
The latter can be computed as an underived pushout, as the bottom map is a cofibration. 
We thus obtain a \cdga{} of the form:
\begin{equation}\label{equation:em-with-apl}
        (\APL(\F^r(X) \times \Gammacon(J^1(\cL) -0)))[y_1,\dotsc,y_r, z_1,\dotsc,z_r]
\end{equation}
with differential given on the extra generators by
\[
    d(y_i) = \APL(\pi_i)(e), \quad d(z_i) = \phi(x_i) - \APL(\pi_i)(c)y_i
\]
with $e, c \in \APL(X)$ cocycles representing $e(\Omega^1(\cL))$ and $c_1(\cL)$ respectively. By \cref{sec:recollections-sectionspace}, we have a quasi-isomorphism
\[
    \Symgr(H^{*}(X;\bQ)[-1]) \equivto \APL(\Gammacon(J^1(\cL) -0)) \dpunct.
\]
By \cref{construction:totaro-model}, there is a zigzag of quasi-isomorphisms
\[
    \APL(\F^r(X)) \lequivto \cdots \equivto C_r(X) \dpunct.
\]
Now recall the following manipulation: if $A$ and $B$ are quasi-isomorphic \cdga{}s (through a zigzag) and $a \in A$, $b \in B$ are cocyles representing the same cohomology class, then $(A[x], dx=a)$ and $(B[x], dx=b)$ are quasi-isomorphic (through a zigzag).
Applying this fact inductively, we see that the \cdga{} \cref{equation:em-with-apl} is quasi-isomorphic to a \cdga{} of the form
\[
    (C_r(X) \tensor \Symgr(H^{*}(X;\bQ)[-1])) [y_1, \dotsc, y_r, z_1, \dotsc, z_r] \dpunct,
\]
with differential
\[
    d(y_i) = \pi_i^*(e(\Omega^1(\cL))), \quad \text{and} \quad d(z_i) = \epsilon_i - \pi_i^*(c_1(\cL))y_i \dpunct. \qedhere
\]
\end{proof}

\subsection{Hodge weights}
\label{cdga-weights-proof}

Here we prove \cref{weights-on-stable-cohomology}, ie the isomorphism in \cref{maintheorem:stability} is weight preserving.

\begin{proof}[Proof of \cref{weights-on-stable-cohomology}]
Consider the Leray spectral sequence of the map 
\begin{align*}
    \pi \from \Z^r(\cL) &\lra (\Omega^1(\cL) - 0)^r|_{\F^r(X)} \\
    (f, x_1,\dotsc,x_r) &\longmapsto (\diff f(x_1), \dotsc, \diff f(x_r)) \dpunct.
\end{align*}
Since this map is algebraic, the spectral sequence is of mixed Hodge structures, and computes the correct weights in cohomology. 
But $\pi$ and any restriction $\pi|\pi^{-1}(V)$ for $V \subset (\Omega^1(\cL) - 0)^r|_{\F^r(X)}$ open is also a microfibration. 
By \cref{adapted-microcomparison}, comparing the Leray sheaves of $\pi$ and $\pi_{C^0}$ in 
\[\begin{tikzcd}
	\Z^r(\cL) \dar["\pi"] \rar["j^1"] & \Z^r_{\cC^0}(\cL) \dar["\pi_{\cC^0}"] \\
(\Omega^1(\cL) - 0)^r|_{\F^r(X)} \rar[equal] & (\Omega^1(\cL) - 0)^r|_{\F^r(X)}
\end{tikzcd}\]
we get that $R^q \pi_* \QQ$ is locally constant when $\cL$ is sufficiently jet ample, and hence the Leray spectral sequence for $\pi$ agrees in a range with the Leray--Serre spectral sequence for $\pi_{C^0}$, starting from the $E_2$ page. 
But we can compute the differentials in this Serre spectral sequence using rational models, following Grivel--Halperin--Thomas \cites[Theorem~5.1]{grivel_formes_1979}{halperin_lectures_1983}. 

To use this construction of the Serre spectral sequence and to understand its terms, we need to understand the cohomology of the (homotopy) fiber of the map $\pi_{\cC^0}$ and the monodromy action of $\pi_1((\Omega^1(\cL) - 0)^r|_{\F^r(X)})$ on it.

Let us first show that the monodromy action is trivial. To this end, observe that the pullback square \cref{main-pullback} identifies the fiber of $\pi_{\cC^0}$ with the fiber of $\ev \from \Gamma_{\cC^0}(J^1(\cL) - 0) \times \F^r(X) \to (J^1(\cL) - 0)^r|_{\F^r(X)}$, and the monodromy representation factors through $\pi_1((J^1(\cL) - 0)^r|_{\F^r(X)})$. Let $\vec{v} = ((x_1,v_1),\dotsc,(x_r,v_r))$ be a basepoint in $(J^1(\cL) - 0)^r$ and write $\vec{x} = (x_1,\dotsc,x_r)$. Let
\[
    \ev_{\vec{x}} \colon \Gamma_{\cC^0}(J^1(\cL) - 0) \lra (J^1(\cL) - 0)_{\vec{x}} = \prod_{i = 1}^r \ (J^1(\cL) - 0)_{x_i}, \quad s \longmapsto (s(x_1),\dotsc,s(x_r))
\]
be the pullback of $\ev$ along the inclusion $(J^1(\cL) - 0)_{\vec{x}} \hookrightarrow (J^1(\cL) - 0)^r|_{\F^r(X)}$ of the fiber at $\vec{x}$. Being pulled back, the inclusion between the homotopy fibers $F(\ev_{\vec x}) \hookrightarrow F(\ev)$ is a homotopy equivalence. Recall the standard point-set model for $F(\ev_{\vec x})$ given by:
\[
    F(\ev_{\vec x}) = \set{(f,\gamma_1,\dotsc,\gamma_r)}{f \in \Gamma_{\cC^0}(J^1(\cL) - 0), \ \gamma_i \colon f(x_i) \rightsquigarrow (x_i,v_i) \text{ path in } (J^1(\cL)-0)_{x_i}} \dpunct.
\]
Now we can describe the monodromy action on $F(\ev_{\vec x}) \simeq F(\ev)$. Indeed, a loop at $\vec{v}$ in $(J^1(\cL) - 0)^r|_{\F^r(X)}$ projects down to $r$ loops in $X$ based respectively at $x_1,\dotsc,x_r$. The loop at $x_i$ gives a monodromy map $(J^1(\cL)-0)_{x_i} \to (J^1(\cL)-0)_{x_i}$, and post-composition on the paths in the model of $F(\ev_{\vec x})$ is our monodromy under investigation. Finally we conclude that this monodromy is trivial since $J^1(\cL)$ is a complex, hence oriented, bundle.

We now compute the rational cohomology of the fiber of $\pi_{\cC^0}$. Let $c = c_1(\cL) \in H^2(X;\bQ)$. As we have explained in \cref{rational-model-proof}, the \cdga{} $A_r(X, c)$ is a relative Sullivan extension of $C_r(X)[\alpha_1,\dotsc,\alpha_r]$ which models $\pi_{\cC^0}$. Therefore the quotient \cdga{}
\[
    A_r(X, c) \otimes_{C_r(X)[\alpha_1,\dotsc,\alpha_r]} \bQ \cong \big(\Symgr(H^{*}(X)[-1]) [\eta_1,\dotsc,\eta_r]), \ d\eta_i = s[X] \big)
\]
is a model of the fiber. Its cohomology is given by the graded vector space
\[
    \Symgr\big(H^{*<2n}(X)[-1]\big)[\eta_1',\dotsc,\eta_{r-1}']
\]
where each $\eta_i'$ lives in degree $2n$ and represents the cocycle $[\eta_i - \eta_r]$.

Since we have verified that the monodromy action is trivial, by \cite[Theorem~5.1]{grivel_formes_1979} we get that the Serre spectral sequence for $\pi_{\cC^0}$ has the form
\begin{align*}
    E_2^{p, q} &\cong H^p((\Omega^1(\cL) - 0)^r|_{\F^r(X)}) \tensor H^q(F(\ev))\\
&\cong H^p((\Omega^1(\cL) - 0)^r|_{\F^r(X)}) \tensor \left(\Symgr\big(H^{*<2n}(X)[-1]\big)[\eta_1', \dotsc, \eta_{r-1}']\right)_q \dpunct.
\end{align*}
Finally, we compute the differentials in this spectral sequence from the differential in $A_r(X, c)$. Recall that we write $sb_j$ for the (shifted) basis of $H^{*<2n}(X)[-1]$. Writing $d_k$ for the differential on the $k$th page, we have:
\[
    \forall k, \ d_k(sb_j) = 0 \dpunct; \qquad d_2\eta_i' = \sum_{|b_j| = 2} (\pi_i^* - \pi_r^*)[b_j^\vee] \tensor sb_j
\]
The class $[sb_j]$ has weight $|b_j| + 2$ since it pulls back from a corresponding generator in $H^{|b_j|+1}(U(\cL))$.
Therefore the only way to make these differentials weight preserving is to give the classes $\eta_i'$ weight $2n + 2$.
Now, for any cohomology class $\gamma \in H^*(\Z^r(\cL))$ of homogeneous weight and \emph{in the stable range}, we can choose a representative cycle in this spectral sequence, and then a subsequent representative in $A_r(X, c)$, and therefore the weight of this class must agree with that described in the theorem.
\end{proof}

\subsection{Explicit computations}
\label{sec:computations}

In this section we list some computations of the stable cohomology $H^*(A_r(X, c))$ for specific $X$, $c$ and $r \ge 2$.

\begin{example}
Let $X = \PP^1$, $r = 2$ and $c$ the only element of $\PP H^2(\PP^1) = \PP^0$.
Consider the algebra presentation $H^*(\PP^1) = \QQ[x]/x^2$ with $|x| = 2$ to write
\[A_2(X, c) = \Symgr[x_1, x_2, G, s_1, s_x, \alpha_1, \alpha_2, \eta_1, \eta_2]/(x_i^2, x_1G - x_2G) \dpunct,\]
with $|G| = |s_1| = |\alpha_i| = 1$, $|x_i| = |\eta_i| = 2$, $|s_x| = 3$ and differential
\[dG = x_1 + x_2 \dpunct, \qquad d\alpha_i = x_i \dpunct, \qquad d\eta_i = s_x + x_i s_1 - x_i \alpha_i \dpunct.\]
Using \cref{lemma:homological-algebra} and suitable changes of variables one can check that
\[H^*(A_2(X, c)) \isom \Symgr[s, \gamma, \eta, \tau]\]
where $s = [s_1]$, $\gamma = [G - \alpha_1 - \alpha_2]$ are of degree $1$, $\eta = [\eta_1 - \eta_2 + (s_1 + G)(\alpha_1 - \alpha_2) + \alpha_1 \alpha_2]$ is of degree $2$ and $\tau = [\alpha_1 x_1] = [\alpha_2 x_2]$ is of degree $3$.

Since in this case $\Z^r(\cO(d))/\CC^\times \isom \F^d(\PP^1)/\fS_{d - r}$ (with the action leaving $r$ of the coordinates fixed) and \cref{scalar-multiple-leray-hirsch} applies, this can also be obtained degree-by-degree from existing computations of $H^i(\F^d(\PP^1))$ as an $\fS_d$ representation, see eg \cite{lehrer_action_1986} for the closely related\footnote{And indeed equivalent, since the action of $\PGL(2)$ on $\PP^1$ by Möbius transformations is transitive and the induced action on $\F^d (\PP^1)$ behaves like a direct factor in cohomology once $d \ge 3$.} computation of $H^i(\F^d(\CC))$.
\end{example}

\begin{example}
\Cref{betti-number-table} lists the stable Betti numbers of $Z^r(\cL)$ for a few $X$ of small dimension, small $r$ and low cohomological degrees.
These were computed using the Sage Mathematics Software System \cite{sage}, specifically its \texttt{commutative-dga} module.
\end{example}

\begin{table}\centering
\begin{tabular}{RCCCC}
\toprule
i & X =  \PP^2,\ r = 2 & X = \Sigma_1,\ r = 2 & X = \PP^1 \times \PP^1,\ r = 2,\ c = [1: 1] & X = \PP^2,\ r = 3 \\\midrule
 0 & 1 &   1 &  1 &  1 \\
 1 & 1 &   5 &  1 &  1 \\
 2 & 2 &  15 &  4 &  3 \\
 3 & 3 &  29 &  6 &  4 \\
 4 & 1 &  47 &  5 &  1 \\
 5 & 4 &  69 & 16 &  9 \\
 6 & 5 &  94 & 14 & 12 \\
 7 & 3 & 122 & 12 &  7 \\
 8 & 4 & 153 & 28 & 15 \\
 9 & 4 & 187 & 18 & 21 \\
10 & 6 & 224 & 15 & 22 \\
\bottomrule
\end{tabular}
\caption{$H^i(A_r(X, c))$ in small examples, with $c$ omitted if $H^2(X) \isom \QQ$. $\Sigma_1$ denotes a Riemann surface of genus $1$.}
\label{betti-number-table}
\end{table}

\begin{example}\label{non-ample-c1-dependence}
If we do not assume that $c \in \PP H^2(X)$ is represented by an element in the ample cone (ie $c_1(\cL)$ for some ample $\cL$) then $H^*(A_r(X, c))$ \emph{can} depend on $c$.
In fact this is visible in perhaps the simplest possible example: $X = \PP^1 \times \PP^1$ and $r = 2$.
If we choose generators $a$ and $b$ of $H^2(\PP^1 \times \PP^1)$ corresponding to the (integral) orientation classes of the two factors, then a class $pa + qb$ (with $p$, $q$ integers), where we write $c = [p: q]$, is ample iff $p, q > 0$.

As listed in \cref{betti-number-table}, $H^9(A_2(\PP^1 \times \PP^1, [1: 1]))$ (ie with $c = [a + b]$) is $18$ dimensional.
However, $H^9(A_2(\PP^1 \times \PP^1, [1: 0])) \isom \QQ^{19}$.
Similarly $H^{10}(A_2(\PP^1 \times \PP^1, [1: 1])) \isom \QQ^{15}$ while $H^{10}(A_2(\PP^1 \times \PP^1, [1: 0])) \isom \QQ^{17}$.

In fact $A_r(\PP^1 \times \PP^1, c)$ with $c = [pa + qb]$ has only two isomorphism classes: one for $p, q \ne 0$ (which are all isomorphic by arguments similar to the proof of \cref{rational-model-stabilizes}), and one where $c = [a]$ or $c = [b]$ (which are isomorphic under swapping factors).
\end{example}

\section{The h-principle}
\label{sec:h-principle}

Let $(x_1,\dotsc,x_r) \in \F^r(X)$ be $r$ distinct points and $v_i \in (J^1(\cL) - 0)_{x_i}$ be non-zero vectors at these points. To lighten the notation, we will abbreviate the whole tuple by
\[
    \vec{v} \defeq \left( (x_1,v_1),\dotsc,(x_r,v_r) \right) \in (J^1(\cL) - 0)^r \dpunct. 
\]
The space of holomorphic sections with derivatives prescribed by $\vec{v}$ is defined to be:
\[
	\Gammahol(\vec{v}) \defeq \set{f \in \Gammahol(\cL)}{j^1(f)(x_i) = (x_i,v_i) \text{ for each } i= 1, \dotsc, r} \dpunct.
\]
This is an affine subspace of the vector space of global sections $\Gammahol(\cL)$ and contains as an open subspace the previously defined
\[
    U(\vec{v}) \defeq \Gammahol(\vec{x}) \cap U(\cL) \dpunct.
\]
Similarly, we have an analogous space defined by continuous sections of $J^1(\cL)$ with prescribed values at the $x_i$:
\[
    U_{\cC^0}(\vec{v}) \defeq \set{s \in \Gamma_{\cC^0}(J^1(\cL) - 0)}{s(x_i) = (x_i, v_i) \text{ for each } i = 1, \dotsc, r} \subset \Gamma_{\cC^0}(J^1(\cL) - 0) \dpunct.
\]
The jet expansion $j^1$ restricts to these subspaces and the goal of this section is to prove the following result:
\begin{theorem}\label{h-principle-prescribed-derivatives}
For a $d$-jet ample line bundle $\cL$, the jet map
\[ j^1 \from U(\vec{v}) \lra U_{\cC^0}(\vec{v}) \]
induces an isomorphism in integral homology in the range of degrees $* < \frac{d-1}{2} - r$.
\end{theorem}

The proof follows that of \cite{aumonier_h-principle_2022} but we could not find a direct way of applying the theorem therein. Indeed, there are two main differences. Firstly, we apply a Vassiliev-style argument to an \emph{affine subspace} $\Gammahol(\vec{v}) \subset \Gammahol(\cL)$ of the vector space of global sections. This is in contrast with \cite{aumonier_h-principle_2022} where the full space of global sections is considered. Secondly, a section $f \in \Gammahol(\vec{v})$ can only be singular at points in $X - \{x_1,\dotsc,x_r\}$ because it has non-vanishing derivatives at the special chosen points $x_i$. As the proofs of \cite[Section~3]{aumonier_h-principle_2022} relied on compactness of $X$ which we lose when removing points, we will need to adapt them to our case.

\subsection{Constructing the Vassiliev spectral sequence}

The space $U(\vec{v})$ is an open subset in the complex affine space $\Gammahol(\vec{v})$ of complex dimension $\dim_\bC \Gammahol(\vec{v})$, whose complement we denote by $\Sigma(\vec{v})$. We topologize them using the canonical topology on the ambient complex affine space. By Alexander duality, there is an isomorphism
\[
    \CCH^i(\Sigma(\vec{v})) \cong \widetilde{H}_{2\dim_\bC \Gammahol(\vec{v})-1-i}(U(\vec{v})).
\]
We want to compute the homology of $U(\vec{v})$, but we will equivalently study the compactly supported \v{C}ech cohomology of its complement. This is technically advantageous: the complement admits a filtration by the number of singularities of a section $f \in \Sigma(\vec{v})$, which allows the construction of a spectral sequence à la Vassiliev. In practice, it is easier to work with an auxiliary filtered space mapping properly down to $\Sigma(\vec{v})$ with acyclic fibers. We construct this space and its associated spectral sequence below.

\bigskip

We write $X^\circ := X - \{x_1,\dotsc,x_r\}$ to denote the punctured space. Define
\[
    \Sing(f) := \set{y \in X^\circ}{f \text{ is singular at } y} \subset X^\circ
\]
to be the singular subspace of a section $f \in \Sigma(\vec{v})$. Let $\catF$ be the category whose objects are the finite sets $[n] \defeq \{0,\dotsc,n\}$ for $n \geq 0$ and whose morphisms are \emph{all} maps of sets $[n] \to [m]$. Let $\catTop$ be the category of topological spaces and continuous maps between them. On objects, define the following functor:
\begin{equation*}
\begin{split}
    \fX \colon \catF^\op &\lra \catTop \\
    [n] &\longmapsto \fX[n] := \set{(f,y_0,\dotsc,y_n) \in \Gammahol(\vec{v}) \times (X^\circ)^{n+1}}{\forall i, \ y_i \in \mathrm{Sing}(f)}
\end{split}
\end{equation*}
where $\fX[n]$ is given the subspace topology from $\Gammahol(\vec{v}) \times (X^\circ)^{n+1}$. On morphisms, for a map of sets $g \colon [n] \to [m]$, we define it by:
\begin{equation*}
\begin{split}
    \fX(g) \colon \fX[m] &\lra \fX[n] \\
    (f,y_0,\dotsc,y_m) &\longmapsto (f,y_{g(0)}, \dotsc, y_{g(n)}).
\end{split}
\end{equation*}
For an integer $k \geq 0$, let $\catF_{\leq k}$ be the full sub-category of $\catF$ on objects $[n]$ for $n \leq k$. Write
\[
    |\Delta^n| = \{ (t_0, \dotsc, t_n) \mid \forall i, 0 \leq t_i \leq 1 \text{ and } t_0 + \cdots + t_n = 1 \} \subset \bR^{n+1}
\]
for the standard topological $n$-simplex, and denote by $\partial |\Delta^n|$ its boundary. In particular, the assignment $[n] \mapsto |\Delta^n|$ gives a functor $\catF \to \catTop$. For an integer $j \geq 0$, we define the \emph{$j$-th geometric realization of $\fX$} by the following coend:
\begin{equation*}
R^j\fX := \int^{[n] \in \catF_{\leq j}} \fX[n] \times |\Delta^n|
= \left( \bigsqcup_{0 \leq n \leq j} \fX[n] \times |\Delta^n| \right) / \sim
\end{equation*}
where the equivalence relation $\sim$ is generated by $(\fX(g)(z),t) \sim (z,g_*(t))$ for all maps $g \colon [n] \to [m]$ in $\catF$. Here $g_* \colon |\Delta^n| \to |\Delta^m|$ denotes the usual map induced on the simplices by functoriality. Note that the main difference between our construction and the geometric realisation of a simplicial space resides in the fact that we allow all maps of sets, in particular the permutations $[n] \to [n]$ are morphisms in $\catF$.

\bigskip

As for simplicial spaces, $R^j\fX$ is obtained from $R^{j-1}\fX$ via a pushout diagram along a subspace, which can be thought of as a kind of latching object. More precisely, define
\begin{equation*}
    L_j \defeq \set{ (f,y_0, \dotsc, y_j) \in \Gammahol(\vec{v}) \times (X^\circ)^{j+1} }{ \exists l \neq k \text{ such that } y_l = y_k } \subset \fX[j]
\end{equation*}
topologized as a subspace of $\fX[j]$, and write $L_j \times_{\fS_{j+1}} |\Delta^j|$ for the quotient space of $L_j \times |\Delta^j|$ by the symmetric group $\fS_{j+1}$ acting on $L_j$ by permuting the singularities $y_i$, and on $|\Delta^j|$ by permuting the coordinates. The following result is immediate from the definitions.
\begin{lemma}[{\cite[Proposition 3.3]{aumonier_h-principle_2022}}]\label{lemma:resolution-skeletal-pushout}
There is a pushout square of topological spaces:
\[
\begin{tikzcd}
	{\left(L_j \times_{\fS_{j+1}} |\Delta^j| \right) \bigcup \left( \fX[j] \times_{\fS_{j+1}} \partial |\Delta^j|\right)} \arrow[d, hook] \arrow[r] & R^{j-1}\fX \arrow[d] \\
	{\fX[j] \times_{\fS_{j+1}} |\Delta^j|} \arrow[r]                                                                                                & R^j\fX
\end{tikzcd}\]
where the left vertical map is a closed embedding. \qed
\end{lemma}
Using the lemma inductively, one sees that the spaces $R^j\fX$ are paracompact and Hausdorff, and that the natural map $R^{j-1}\fX \to R^j\fX$ is a closed embedding. Another direct consequence is the following homeomorphism:
\begin{equation}\label{eqn:differencestepsfiltration}
    R^j\fX - R^{j-1}\fX \cong Y_j \times_{\fS_{j+1}} |\mathring{\Delta^j}|
\end{equation}
where
\begin{equation*}
Y_j := \set{ (f,y_0, \dotsc, y_j) \in \fX[j]}{y_l \neq y_k \text{ if } l \neq k} = \fX[j] - L_j \subset \fX[j]
\end{equation*}
is the subspace of $\fX[j]$ where the singularities are pairwise distinct, and $|\mathring{\Delta^j}|$ is the interior of the simplex.
The next lemma is stated in \cite[Lemma~3.5]{aumonier_h-principle_2022} but its proof needs to be adapted here to account for the fact that $X^\circ$ is not compact.
\begin{lemma}\label{lemma:rhonisproper}
For any $n \geq 0$, the map $\rho_n \colon \fX[n] \to \Gammahol(\vec{v})$ given by $(f,y_0,\dotsc,y_n) \mapsto f$ is a proper map.
\end{lemma}
\begin{proof}
Let $K \subset \Gammahol(\vec v)$ be compact. Then
\[\set{(f, x) \in K \times X}{x \notin \Sing(f)}\]
is open in $K \times X$ and contains $K \times \{x_i\}$ for $i = 1, \dotsc, r$.
By the tube lemma, it must contain some $K \times V$, where $V \subset X$ is a neighborhood of $\{x_1, \dotsc, x_r\}$.
Then $\rho_n^{-1}(K)$ is a closed subset of $K \times (X - V)^{n + 1}$ and hence is compact.
\end{proof}
The natural projections maps $\fX[n] \times |\Delta^n| \to \fX[n] \tolabel{\rho_n} \Gammahol(\vec{v})$ give rise to a map from the geometric realization $\tau_j \colon R^j\fX \to \Sigma(\vec{v})$. The proof of \cite[Lemma~3.6]{aumonier_h-principle_2022} using the adapted Lemma~\ref{lemma:rhonisproper} above then shows that:
\begin{lemma}
For any integer $j \geq 0$, the map $\tau_j \colon R^j\fX \to \Sigma(\vec{v})$ is proper. \qed
\end{lemma}

\bigskip

We are now half-way through the construction of a replacement of $\Sigma(\vec{v})$: the space $R^j\fX$ maps properly down to it, but the fibers are not all acyclic. Indeed, the fiber $\tau_j^{-1}(f)$ above a section $f \in \Sigma(\vec{v})$ that has at most $j+1$ singularities is a (possibly degenerate) $j$-simplex whose vertices are indexed by the singular points of $f$. Hence its homology vanishes in positive degrees. On the contrary, if $f$ has at least $j+2$ singularities the fiber is not contractible nor acyclic in general. This issue is fixed by the following construction.
Let $N \geq 0$ be an integer and let
\[
    \Sigma(\vec{v})_{\geq N+2} := \set{ f \in \Gammahol(\vec{v}) }{ \#\Sing(f)\geq N+2 } \subset \Sigma(\vec{v})
\]
be the subspace of those sections with at least $N+2$ singular points. We denote by $\overline{\Sigma(\vec{v})_{\geq N+2}}$ its closure in $\Sigma(\vec{v})$. Define $R^N_\text{cone}\fX$ by the following homotopy pushout:
\begin{equation*}
\begin{tikzcd}
\tau_N^{-1}\left(\overline{\Sigma_{\geq N+2}}\right) \arrow[d, "\tau_N"'] \arrow[r, hook] \arrow[dr, phantom, "^\mathrm{ho}\ulcorner", very near end] & R^N\fX \arrow[d]      \\
\overline{\Sigma_{\geq N+2}} \arrow[r]                                           & R^N_{\text{cone}}\fX.
\end{tikzcd}
\end{equation*}
The three other spaces map to $\Sigma(\vec{v})$, thus yielding a map from the homotopy pushout $\pi \colon R^N_\text{cone}\fX \to \Sigma(\vec{v})$.

\begin{proposition}[{\cite[Lemma~3.8, Lemma~3.9, Proposition~3.10]{aumonier_h-principle_2022}}]\label{prop:acyclicresolution}
The space $R^N_\text{cone}\fX$ is paracompact, locally compact and Hausdorff. The map $\pi \colon R^N_\text{cone}\fX \to \Sigma(\vec{v})$ is proper and induces an isomorphism in cohomology with compact supports. \qed
\end{proposition}

\bigskip

Using the closed embeddings $R^{j-1}\fX \subset R^j\fX$ obtained in \cref{lemma:resolution-skeletal-pushout}, we define the following filtration on $R^N_{\text{cone}}\fX$:
\[
    F_0 = R^0\fX \subset F_1 = R^1\fX \subset \cdots \subset F_N = R^N\fX \subset F_{N+1} = R^N_{\text{cone}}\fX.
\]
By standard arguments on spectral sequences associated to filtered complexes, we obtain:
\begin{proposition}[{\cite[Proposition~3.11]{aumonier_h-principle_2022}}]\label{prop:existenceVassilievSS}
There is a spectral sequence on the first quadrant $s,t \geq 0$:
\[
E_1^{s,t} = \CCH^{s+t}(F_s - F_{s-1}; \bZ) \implies \CCH^{s+t}(R^N_\text{cone}\fX; \bZ) \cong \widetilde{H}_{2\dim_\bC \Gammahol(\vec{v}) - 1 - s - t}(U(\vec{v}); \bZ).
\]
The differential $d^r$ on the $r$-th page of the spectral sequence has bi-degree $(r,1-r)$, ie it is a morphism $d_r^{s,t} \colon E_r^{s,t} \to E_r^{s+r,t-r+1}$. \qed
\end{proposition}

\subsection{Analyzing the spectral sequence}

In this section, we describe the terms on the $E_1$-page of the spectral sequence given in Proposition~\ref{prop:existenceVassilievSS}. The constructions of the previous section and the spectral sequence depend on an integer $N$ that we are a priori free to choose. We will follow the following convention for the remainder of this article:
\begin{convention}\label{convention:bigN}
Let $N$ be the largest integer such that $\cL$ is $(2(r+N+1) -1)$-jet ample.
\end{convention}

We will consider two cases separately:
\begin{enumerate}
\item When $0 \leq s \leq N$, we have $F_s - F_{s-1} = R^s\fX - R^{s-1}\fX$ whose cohomology can be understood using the homeomorphism~\eqref{eqn:differencestepsfiltration}.
\item When $s = N+1$, we have $F_s - F_{s-1} = R^N_\text{cone}\fX - R^N\fX$ whose cohomology we will bound.
\end{enumerate}
Note that outside the band $0 \leq s \leq N+1$, all the terms $E_1^{s,t}$ vanish as the filtration giving rise to the spectral sequence is indexed from $0$ to $N+1$.

\subsubsection{Cohomology in the columns \texorpdfstring{$0 \leq s \leq N$}{0 ≤ s ≤ N}}

We first treat the cases where $0 \leq s \leq N$. Let us recall the homeomorphism~\eqref{eqn:differencestepsfiltration}:
\[
F_s - F_{s-1} = R^s\fX - R^{s-1}\fX \cong \set{ (f,y_0, \dotsc, y_s) \in \fX[s] }{ y_l \neq y_k \text{ if } l \neq k } \times_{\fS_{s+1}} \mathring{|\Delta^s|}
\]
where the symmetric group $\fS_{s+1}$ acts on the left by permuting the singular points $y_i$ and on the right by permuting the vertices of the simplex. There is a natural map down to the configuration space $\F^{s+1}(X^\circ)$ which forgets the simplex and the section $f$. An argument similar to that of \cite[Section~4.1]{aumonier_h-principle_2022} then shows that this map is a fiber bundle.
\begin{lemma}
For $0 \leq s \leq N$, the natural map
\begin{equation*}
\begin{split}
	F_s - F_{s-1} \cong \set{ (f,y_0, \dotsc, y_s) \in \fX[s] }{ y_l \neq y_k \text{ if } l \neq k } \times_{\fS_{s+1}} \mathring{|\Delta^s|} &\lra \F^{s+1}(X^\circ) \\
	[ (f,y_0,\dotsc,y_s), \lambda ] &\longmapsto \left\{ y_0, \dotsc, y_s \right\}
\end{split}
\end{equation*}
is a fiber bundle. The fiber above a point $\vec{y} = \left\{ y_0, \dotsc, y_s \right\} \in \F^{s+1}(X^\circ)$ is the product
\[
\mathring{|\Delta^s|} \times \left\{ f \in \Gammahol(\vec{v}) \mid \vec{y} \subset \Sing(f) \right\},
\]
where the right-hand term is the complex affine subspace of $\Gammahol(\vec{v})$ of those sections having all points in $\vec{y}$ as singularities. It is of complex dimension $\dim_\bC \Gammahol(\vec{v}) - (s+1)(\dim_\bC X +1)$.
\end{lemma}

\begin{proof}[Proof sketch]
The main point is to see that the affine spaces $\set{ f \in \Gammahol(\vec{v}) }{ \vec{y} \subset \Sing(f) }$ all have the same dimension regardless of $\vec{y} \in \F^{s+1}(X^\circ)$. They are given by $(s+1)(\dim_\bC X +1)$ equations: a section $f$ is singular at $y_i$ when both the value $f(y_i)$ and all the partial derivatives of $f$ at $y_i$ vanish. In total, this imposes $1 + \dim_\bC X$ equations on $f$. The lemma is then proven if these equations are linearly independent. The line bundle $\cL$ is $(2(r+s+1) -1)$-jet ample by \cref{convention:bigN}, so the evaluation map
\[
\Gammahol(\cL) \lra \bigoplus_{i = 1}^r J^1(\cL)_{x_i} \ \oplus \ \bigoplus_{j = 0}^s J^1(\cL)_{y_i}
\]
is surjective. This directly implies that the considered equations are linearly independent.
\end{proof}

\begin{remark}
The proof above follows very closely the one given for \cite[Lemma~4.2]{aumonier_h-principle_2022}, except for the fact that we also impose derivatives at the fixed $x_i$ to get the existence of sections in $\Gammahol(\vec{v})$. This is reflected by the appearance of the constant $r$ in the required jet ampleness of $\cL$.
\end{remark}

Applying the Thom isomorphism for cohomology with compact supports yields:
\begin{lemma}[{Compare \cite[Proposition~4.3]{aumonier_h-principle_2022}}]
For $0 \leq s \leq N$, we have an isomorphism
\[
E_1^{s,t} \cong \CCH^{t - 2\dim_\bC \Gammahol(\vec{v}) + 2(s+1)(\dim_\bC X+1)}(\F^{s+1}(X^\circ); \bZ^\mathrm{sign})
\]
where $\bZ^\mathrm{sign}$ is the local coefficients system on the configuration space given by the sign representation. \qed
\end{lemma}

\subsubsection{Cohomology in the column \texorpdfstring{$s = N+1$}{s = N + 1}}

We now turn our attention to the last column on the first page of the spectral sequence. In this case $s = N+1$ and the groups are
\[
    E_1^{N+1,t} = \CCH^{N+1+t}(R^N_\text{cone}\fX - R^N\fX; \bZ).
\]
We shall show that these group vanish for $t$ big enough. More precisely, following \cite[Section~4.2]{aumonier_h-principle_2022} we obtain:
\begin{lemma}[{Compare \cite[Proposition~4.9]{aumonier_h-principle_2022}}]
For $t > 2\dim_\bC \Gammahol(\vec{v}) - 2N - 2$ we have
\[
E_1^{N+1,t} = \CCH^{N+1+t}(R^N_\text{cone}\fX - R^N\fX; \bZ) = 0.
\]
\end{lemma}
\begin{proof}[Proof sketch]
The proof of \cite[Proposition~4.9]{aumonier_h-principle_2022} actually applies to this situation, but we sketch the main ideas in the particular case at hand. By construction, the space $R^N_\text{cone}\fX - R^N\fX$ is the fiberwise (for the map $\tau_N \colon R^N\fX \to \Sigma(\vec{v})$) open cone over $\overline{\Sigma(\vec{v})_{\geq N+2}}$. It can be stratified by
\begin{align*}
\mathrm{Str}_{-1} & := \overline{\Sigma(\vec{v})_{\geq N+2}},                                                                                                                                                \\
\mathrm{Str}_0    & := \left( \tau_N^{-1}\left(\overline{\Sigma(\vec{v})_{\geq N+2}}\right) \times (0,1) \right)  \cap  \left( R^0\fX \times (0,1)\right),                                                   \\
\mathrm{Str}_j    & := \left( \tau_N^{-1}\left(\overline{\Sigma(\vec{v})_{\geq N+2}}\right) \times (0,1) \right)  \cap  \left( (R^j\fX - R^{j-1}\fX) \times (0,1)\right) \quad \text{ for } 1 \leq j \leq N.
\end{align*}
Furthermore, the homeomorphism~\eqref{eqn:differencestepsfiltration} shows that for $0 \leq j \leq N$ we have a homeomorphism
\[
\mathrm{Str}_j \cong \left( Y_j^{\geq N+2} \times_{\fS_{j+1}} \mathring{|\Delta^j|} \right) \times (0,1)
\]
where
\[
Y_j^{\geq N+2} := \set{ (f,y_0,\dotsc,y_j) \in \Gammahol(\vec{v}) \times \F^{j+1}(X^\circ) }{ f \in \overline{\Sigma(\vec{v})_{\geq N+2}} \text{ and } y_i \in \mathrm{Sing}(f) }.
\]
One sees that this latter space $Y_j^{\geq N+2}$ is a real semi-algebraic set and that the natural forgetful map
\[
\set{ (f,y_0, \dotsc, y_N) \in \Sigma(\vec{v})_{\geq N+2} \times \F^{N+1}(X^\circ) }{ y_i \in \mathrm{Sing}(f)  } \lra Y_j^{\geq N+2}
\]
is algebraic and has dense image. This implies that the dimension of $Y_j^{\geq N+2}$ is at most that of the space on the left-hand side. One computes it to be at most $2\dim_\bC \Gammahol(\vec{v}) - 2(N+1)$ (see \cite[Lemma~4.8]{aumonier_h-principle_2022}), implying that all the strata have dimension at most $2\dim_\bC \Gammahol(\vec{v}) - N - 1$. Therefore the compactly supported cohomology of their union vanishes above this dimension, ie $\CCH^{N+1+t}(R^N_\text{cone}\fX - R^N\fX; \bZ) = 0$ whenever $N+1+t > 2\dim_\bC \Gammahol(\vec{v}) - N - 1$.
\end{proof}

\subsection{From holomorphic to continuous sections}

In the previous sections, we have constructed a spectral sequence converging to the homology of $U(\vec{v})$ and have described some features of its first page. We would like to do the same for $U_{\cC^0}(\vec{v})$ as well as provide a morphism of spectral sequences that is an isomorphism in a range on the first page. But the space of continuous sections of $J^1(\cL)$ is not finite dimensional, hence Alexander duality cannot be applied directly. This problem can be remedied by introducing a growing filtration
\[
U(\vec{v}) \tolabel{j^1}  U_0(\vec{v}) \lra U_1(\vec{v}) \lra \cdots \lra \colim\limits_{k \to \infty} U_k(\vec{v}) \equivto U_{\cC^0}(\vec{v})
\]
where every map $U_k(\vec{v}) \to U_{k+1}(\vec{v})$ is shown to be a homology isomorphism in a range using a spectral sequence similar to the one above, and the colimit of the $U_k(\vec{v})$ is homotopy equivalent to $U_{\cC^0}(\vec{v})$.

\subsubsection{Definition of the filtration}

We follow \cite[Section~5]{aumonier_h-principle_2022} to describe roughly how the spaces $U_k(\vec{v})$ are constructed, but refer to that article for the full details. The main idea is to consider the complex conjugate (or equivalently the dual) line bundle $\overline{\cL}$ of $\cL$. Taking the complex conjugates of the values of a section gives an $\bR$-linear morphism:
\[
\overline{\cdot} \colon \Gamma_{\cC^0}(\cL) \lra \Gamma_{\cC^0}(\overline{\cL}).
\]
For a complex vector space $V$, we will denote by $\overline{V}$ the complex vector space with the same underlying abelian group, but where the $\bC$-module structure is given by multiplication by the complex conjugate. Complex conjugation thus gives a $\bC$-linear morphism $\overline{\Gammahol(\cL)} \to \Gamma_{\cC^0}(\overline{\cL})$. As the tensor product $\cL \otimes \overline{\cL}$ is the trivial line bundle $X \times \bC$ one can consider the multiplication map:
\begin{equation}\label{eqn:multiplicationbycomplexconjugate}
\Gammahol(\cL) \otimes_\bC \overline{\Gammahol(\cL)} \subset \Gamma_{\cC^0}(\cL) \otimes_\bC \overline{\Gamma_{\cC^0}(\cL)} \lra \Gamma_{\cC^0}(\cL \otimes \overline{\cL}) \cong \Gamma_{\cC^0}(X \times \bC).
\end{equation}
Likewise, for any integer $k \geq 0$, one gets a multiplication map
\[
\mu_k \colon \Gammahol((J^1(\cL)) \otimes \cL^k) \otimes_\bC \overline{\Gammahol(\cL^k)} \lra \Gamma_{\cC^0}((J^1(\cL)) \otimes \cL^k \otimes \overline{\cL^k}) \cong \Gamma_{\cC^0}(J^1(\cL)).
\]

\begin{definition}
For any integer $k \geq 0$, we define
\[
\Gamma_k(\vec{v}) := \mu_k^{-1}(\Gamma_{\cC^0}(\vec{v}))
\]
and
\[
U_k(\vec{v}) := \mu_k^{-1}(U_{\cC^0}(\vec{v})). \qedhere
\]
\end{definition}
Importantly for us, the space $\Gamma_k(\vec{v})$ is an affine subspace of the finite dimensional complex vector space $\Gammahol((J^1(\cL)) \otimes \cL^k) \otimes_\bC \overline{\Gammahol(\cL^k)}$, hence is itself finite dimensional. We now describe the maps $U_k(\vec{v}) \to U_{k+1}(\vec{v})$. Using the triviality of $\cL \otimes \overline{\cL}$, we can choose an element
\[
\eta \in \Gammahol(\cL) \otimes_\bC \overline{\Gammahol(\cL)}
\]
corresponding to the constant function with value $1$ under the multiplication map~\eqref{eqn:multiplicationbycomplexconjugate}. Multiplying sections by this element gives a commutative square:
\[
\begin{tikzcd}
\Gammahol((J^1(\cL)) \otimes \cL^k) \otimes_\bC \overline{\Gammahol(\cL^k)} \arrow[d, "\cdot \eta"'] \arrow[r, "\mu_k"] & \Gamma_{\cC^0}(J^1(\cL)) \arrow[d, equal] \\
\Gammahol((J^1(\cL)) \otimes \cL^{k+1}) \otimes_\bC \overline{\Gammahol(\cL^{k+1})} \arrow[r, "\mu_{k+1}"']             & \Gamma_{\cC^0}(J^1(\cL))
\end{tikzcd}\]
This commutativity readily implies that the left vertical map restricts to a map
\[
\cdot \eta \colon U_k(\vec{v}) \lra U_{k+1}(\vec{v}).
\]

\subsubsection{Comparing spectral sequences}

As explained in \cite[Section~5.4]{aumonier_h-principle_2022}, the construction of the spectral sequence and the analysis of its first page can be carried out for the spaces $U_k(\vec{v}) \subset \Gamma_k(\vec{v})$. We summarize the results here:
\begin{proposition}[{Compare \cite[Proposition~5.5]{aumonier_h-principle_2022}}]
Let $N$ be as in \cref{convention:bigN}. For any integer $k \geq 0$, there is a cohomologically-indexed spectral sequence supported on the strip $0 \leq s \leq N+1$ and $t\geq 0$:
\[
E_1^{s,t} \implies \widetilde{H}_{2\dim_\bC \Gamma_k(\vec{v}) - 1 - s - t}(U_k(\vec{v}); \bZ).
\]
When $0 \leq s \leq N$, there is an isomorphism
\[
E_1^{s,t} \cong \CCH^{t - 2\dim_\bC \Gamma_k(\vec{v}) + 2(s+1)(\dim_\bC X+1)}(\F^{s+1}(X^\circ); \bZ^\mathrm{sign})
\]
where $\bZ^\mathrm{sign}$ is the local coefficients system on the configuration space given by the sign representation.
And when $s = N+1$ and $t > 2\dim_\bC \Gamma_k(\vec{v}) - 2N -2$ we have
\[
\pushQED{\qed} 
    E_1^{N+1,t} = 0. \qedhere
\popQED
\]
\end{proposition}
For any $k \geq 0$ the groups on the first page in the range $0 \leq s \leq N$ are all given in terms of the cohomology of configuration spaces of points in $X^\circ$. The only subtle difference is that these groups are indexed differently, as the degree shift depends on the dimension of $\Gamma_k(\vec{v})$ which varies with $k$. This degree shift is also apparent on the abutment of the spectral sequence. Thus, if we could construct a morphism of spectral sequences shifting the total degree by $\dim_\bC \Gamma_{k+1}(\vec{v}) - \dim_\bC \Gamma_k(\vec{v})$, we would obtain on the abutment a morphism
\[
\widetilde{H}_*(U_k(\vec{v}); \bZ) \lra \widetilde{H}_*(U_{k+1}(\vec{v}); \bZ).
\]
Suppose furthermore that we could construct this morphism of spectral sequences such that it were an isomorphism on the $E_1^{s,t}$ groups when $0 \leq s \leq N$ (which we recall are equal up to this degree shift). Then the vanishing result in the column $s = N+1$ would imply that the morphism on the abutment would be an isomorphism in the range of degrees $* < N$.

In \cite[Section~6]{aumonier_h-principle_2022}, it is explained how to construct this morphism such that the induced morphism $\widetilde{H}_*(U_k(\vec{v}); \bZ) \to \widetilde{H}_*(U_{k+1}(\vec{v}); \bZ)$ is the one induced by $\cdot \eta \colon U_k(\vec{v}) \to U_{k+1}(\vec{v})$ in homology. Likewise, the argument works for the jet map $j^1 \colon U(\vec{v}) \to U_0(\vec{v})$. To sum up, we have:
\begin{proposition}[{Compare \cite[Proposition~6.6]{aumonier_h-principle_2022}}]\label{prop:homologyisoinarange}
Let $N$ be as in \cref{convention:bigN}. Let $k \geq 0$ be an integer. The map $\cdot \eta \colon U_k(\vec{v}) \to U_{k+1}(\vec{v})$ induces an isomorphism in homology in the range of degrees $* < N$. Similarly, the jet map $j^1 \colon U(\vec{v}) \to U_0(\vec{v})$ induces an isomorphism in homology in the same range. \qed
\end{proposition}

\subsubsection{The Stone--Weierstrass theorem}

In view of Proposition~\ref{prop:homologyisoinarange}, it suffices to show that
\[
\colim\limits_{k \to \infty} U_k(\vec{v}) \lra U_{\cC^0}(\vec{v})
\]
is a weak homotopy equivalence to finish the proof of Theorem~\ref{h-principle-prescribed-derivatives}. The proof is analogous to that given in \cite[Section~7]{aumonier_h-principle_2022}, but using the following version of the Stone--Weierstrass theorem with interpolation.
\begin{theorem}[Stone--Weierstrass]\label{thm:StoneWeiestrass}
Let $E \to B$ be a finite rank real vector bundle over a compact Hausdorff space. Let $A \subset \cC^0(B,\bR)$ be a subalgebra, $\{s_j\}_{j\in J}$ be a set of sections, and $\cA$ be the $A$-module generated by the $s_j$. Let $P = \{b_1,b_2,\dotsc,b_k\} \subset B$ be a finite (possibly empty) set of distinct points, and $V = \{v_1,v_2,\dotsc,v_k\} \subset E$ be vectors $v_i \in E|_{b_i}$ in the fibers above the $b_i$. Define
\[
    \cA^{P,V} := \set{ f \in \cA }{ \forall i, \ f(b_i) = v_i } \subset \cA
\]
and
\[
    \Gamma_{\cC^0}^{P,V}(E) := \set{ f \in\Gamma_{\cC^0}(E) }{ \forall i, \ f(b_i) = v_i} \subset \Gamma_{\cC^0}^{P,V}(E)
\]
to be the subsets of sections with prescribed values at the $b_i$. Suppose that
\begin{enumerate}
    \item the subalgebra $A$ separates the points of $B$: for any $x,y \in B$, there exists $h \in A$ such that $h(x) \neq h(y)$;
    \item for any $x \in B$, there exists $h \in A$ such that $h(x) \neq 0$;
    \item for any $x\in B$, the fiber $E_x$ is spanned by the $s_j(x)$ as an $\bR$-vector space.
\end{enumerate}
Then $\cA^{P,V}$ is dense for the sup-norm (induced by the choice of any inner product on $E$) in the space $\Gamma_{\cC^0}^{P,V}(E)$.
\end{theorem}
\begin{proof}[Sketch of a proof] The theorem follows from the original Stone--Weierstrass theorem for functions from a compact space to the real line when the set $P$ is empty, and its variations allowing interpolation in general (see e.g. \cite[Theorem~1]{deutsch_simultaneous_1966}). Indeed, by compactness, we may find a finite number of sections $s_1, \dotsc, s_n$ such that $s_1(x), \dotsc, s_n(x)$ span the fiber at each $x\in B$. Then $\cA$ contains all sections of the form $a_1 s_1 + \cdots + a_n s_n$ for $a_i \in A$, and every continuous section of $E$ can be written as $f_1 s_1 + \cdots + f_n s_n$ with $f_i \in \cC^0(B,\bR)$. We may finally use the usual Stone--Weierstrass theorem, or its adaptation with interpolation, for the functions $f_i$.  
\end{proof}

\printbibliography

\end{document}